\documentclass[reqno,10pt]{amsart}
\usepackage[english]{babel}
\usepackage{amsthm}
\usepackage{amsmath}
\usepackage{amssymb}
\usepackage{mathrsfs}
\usepackage[latin1]{inputenc}
\usepackage{tikz}
\usetikzlibrary{shapes,arrows}
\usepackage{float}

\usepackage[normalem]{ulem}

\usepackage[all]{xy}
\usepackage{dsfont}
\usepackage[bookmarksnumbered,colorlinks]{hyperref}
\usepackage{epsf}
\usepackage{enumerate}
\def\bb#1\eb{\textcolor{blue}
{#1}} %
\newcommand{\R}{\mathds R}

\newcommand{\be}{\begin{equation}}
\newcommand{\ee}{\end{equation}}

   \def\br#1\er{\textcolor{red}{#1}} %
      \def\bb#1\eb{\textcolor{blue}{#1}} %
        \def\bg#1\eg{\textcolor{purple}{#1}} %

\title[Zermelo navigation in pseudo-Finsler metrics]{Zermelo navigation in pseudo-Finsler metrics}

\author[M. A. Javaloyes]{Miguel Angel Javaloyes}
\address{Departamento de Matem\'aticas, \hfill\break\indent
Universidad de Murcia, \hfill\break\indent
Campus de Espinardo,\hfill\break\indent
30100 Espinardo, Murcia, Spain}
\email{majava@um.es}

\author[H. Vit\'orio]{Henrique Vit\'orio}
\address{Departamento de Matem\'atica, \hfill\break\indent
Universidade Federal de Pernambuco, \hfill\break\indent
Cidade Universit\'aria, \hfill\break\indent
Recife, Pernambuco, Brazil}
\email{henriquevitorio@dmat.ufpe.br}

\date{1.12.2011}


\begin{document}
\newtheorem{thm}{Theorem}[section]
\newtheorem{prop}[thm]{Proposition}
\newtheorem{lemma}[thm]{Lemma}
\newtheorem{cor}[thm]{Corollary}
\theoremstyle{definition}
\newtheorem{defi}[thm]{Definition}
\newtheorem{notation}[thm]{Notation}
\newtheorem{exe}[thm]{Example}
\newtheorem{conj}[thm]{Conjecture}
\newtheorem{prob}[thm]{Problem}
\newtheorem{rem}[thm]{Remark}
\newtheorem{conv}[thm]{Convention}
\newtheorem{crit}[thm]{Criterion}
\newtheorem{propdef}[thm]{Proposition-definition}
\newtheorem{lemmadef}[thm]{Lemma-definition}

\begin{abstract}
We generalize the notion of Zermelo navigation to arbitrary pseudo-Finsler metrics possibly defined in conic subsets. The translation of a pseudo-Finsler metric $F$ is a new pseudo-Finsler metric whose indicatrix is the translation of the indicatrix of $F$ by a vector field $W$ at each point, where $W$ is an arbitrary vector field.  Then we show that the Matsumoto tensor of a pseudo-Finsler metric  is equal to zero if and only if it is  the translation of a semi-Riemannian metric,  and when $W$ is homothetic, the flag curvature of the translation coincides with the one of the original one up to the addition of a non-positive constant. In this case, we also give a description of the geodesic flow of the translation.
\end{abstract}

\maketitle
\section{Introduction}

The recent appearance of pseudo-Finsler metrics in the formulation of certain modern physical theories (see for instance \cite{kostelecky11}) has attracted some attention to the study of such metrics. In this paper, we will be concerned with pseudo-Finsler metrics in a generalization of the Zermelo navigation problem. The original problem, proposed by Zermelo in 1931 \cite{Ze31}, aims to describe the trajectories that minimize the time in the presence of a wind or current assuming that the speed of the body is constant without the wind. It was observed by Z. Shen \cite{Sh03} that when the wind does not depend on time and the wind is mild, the minimizing trajectories can be described as geodesics of a Finsler metric. Indeed, when the trajectories without wind are described by geodesics of a Riemannian metric $g$, the minimizing time trajectories in the presence of a wind are geodesics of a Finsler metric whose set $\hat{\Sigma}$ of unit tangent vectors is, at each point, the translation by the wind $W$ of the set $\Sigma$ of unit tangent vectors of $g$:
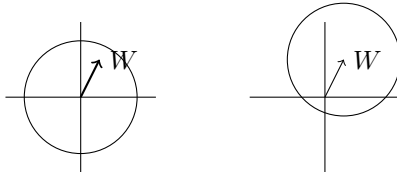
\begin{figure}[H]
\begin{minipage}{.49\linewidth}
\begin{tikzpicture}
\draw (-1,0) -- (1,0);
\draw (0,-1) -- (0,1);
\draw (0,0) circle (0.75cm);
\draw[->,thick] (0,0) -- (0.25,0.5) node[below,right]{$W$};
\end{tikzpicture}
\hspace{1cm}
\begin{tikzpicture}
\draw (-1,0) -- (1,0);
\draw (0,-1) -- (0,1);
\draw (0.25,0.5) circle (0.75cm);
\draw[->] (0,0) -- (0.25,0.5) node[below,right]{$W$};
\end{tikzpicture}
\end{minipage}
\caption{Indicatrix translated by $W$ with $F(-W)<1$}\label{mildTrans}
\end{figure}
\noindent Now, if we consider strong winds, then two pseudo-Finsler metrics will emerge, one with positive-definite fundamental tensor, and the other one of Lorentz type. Moreover, both are defined in the same conic convex region (see Figure \ref{strongTrans2}).
This kind of situation occurs, for instance, in the study of causality of space-times (see \cite{CJS14}).
\begin{figure}
\centering
\begin{tikzpicture}
   \draw [blue,thick,domain=-30:210] plot ({cos(\x)}, {sin(\x)});
   \draw (0,-2) -- (1.732,1);
   \draw (0,-2) -- (-1.732,1);
   \draw[->,thick] (0,-2) -- (0,0) node[below,right]{$W$};
\end{tikzpicture}\hspace{1cm}
\begin{tikzpicture}
   \draw [blue,thick,domain=210:330] plot ({cos(\x)}, {sin(\x)});
   \draw (0,-2) -- (1.732,1);
   \draw (0,-2) -- (-1.732,1);
   \draw[->,thick] (0,-2) -- (0,0) node[below,right]{$W$};
\end{tikzpicture}
\caption{When $F(-W)>1$ there are two translated indicatrices}\label{strongTrans2}
\end{figure}
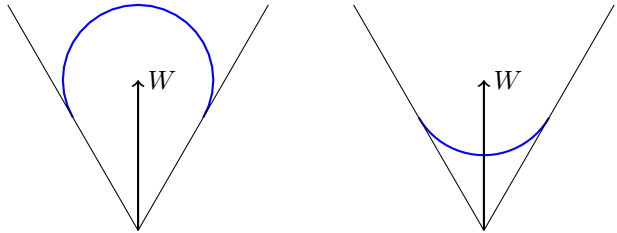

In this paper we will consider the general problem of navigation on a conic pseudo-Finsler manifold $(M,F)$. In this setting, we will say that a pseudo-Finsler metric $\hat{F}$ is a translation of $F$ by a vector field $W$ if, at each point $p\in M$, its indicatrix is (a connected open subset of) the translation by $W_p$ of the indicatrix of $F$. In the case $F$ is semi-Riemannian, the translated metrics will consist of the metrics that, at each point, are of Randers or Kropina types (compare with \cite[Proposition 3.1]{BiJa11} or \cite{BCS04} in the Randers case, and with \cite[Proposition 2.40]{CJS14} for wind Riemannian structures)
\begin{equation}\nonumber
\pm\sqrt{h(v,v)}+\beta(v)~~~{\rm and}~~~\frac{h(v,v)}{\beta(v)},
\end{equation}
where $h$ is a semi-Riemannian metric and $\beta$ a one-form, being thus called pseudo-Randers-Kropina metrics (see $\S$2.3 for a precise description of these metrics).

In $\S$3 we define the Matsumoto tensor of a conic pseudo-Finsler metric. This definition differs from the original one by the appearance of a sign depending on the index of the fundamental tensor. We then extend to these metrics a classic theorem by Matsumoto and Hojo \cite{MaHo78}, proving
\begin{thm}\label{thmmatsumoto}
A conic pseudo-Finsler manifold $(M,F)$ of dimension at least $3$ is of pseudo-Randers-Kropina type if, and only if, its Matsumoto tensor vanishes identically.
\end{thm}
\noindent Our proof (see $\S$3.3)
will consist in adapting to the pseudo-Finsler setting a very geometric proof by Mo and Huang \cite{MoHu10} which makes use of affine differential geometry.

In $\S$4 we will consider the Legendre dual of a conic pseudo-Finsler metric from a more geometric viewpoint. It is remarkable that, as in the Finsler case, the process of translating a conic pseudo-Finsler metric has a nice dual description. Indeed, we show in Proposition \ref{propdichotomy}  that the Legendre duals of $F$ and $\hat{F}$ are related by
\begin{equation}\nonumber
\hat{F}^*(\xi)=F(\xi)+W(\xi)~~~{\rm or}~~~\hat{F}^*(\xi)=-\tilde{F}^*(\xi)+W(\xi),
\end{equation}
where $\tilde{F}$ is the reverse metric of $F$, according to the translation being straight or reverse (see Definition \ref{CharacTrans}).
 The relation between the Legendre duals of $F$ and $\hat{F}$ allows us to relate the co-geodesic flows of $F$ and $\hat{F}$ in some special cases, as 
 noted by Ziller in \cite{Ziller} in the case where $W$ is a Killing field on the sphere. So, in $\S$5 we consider the relation between the geodesics and between the flag curvatures of $F$ and $\hat{F}$ in the case $\hat{F}$ is a translation of $F$ by a homothetic vector field $W$, which means that
\begin{equation}\label{homothetic}
(\psi^W_t)^*F={\rm e}^{-\sigma t}F,~~~\mbox{for some $\sigma\in\mathbb{R}$},
\end{equation}
where $\psi_t^W$ is the flow of $W$.
 In this case we are able to apply a simple, but useful, general result, Lemma \ref{lemmaflow}, to relate the co-geodesic flows of $F$ and $\hat{F}$,
providing a simple proof of
\begin{thm}\label{geodesicflow}
The unit speed geodesics of $(M,\hat{F})$ can be expressed (at least in a neighborhood of $t=0$) as
$\hat\gamma(t)=\psi_t^W(\gamma(f(t)))$, where
\[f(t)=\begin{cases}
\frac{e^{\sigma t}-1}{\sigma}, &\text{if $\sigma\not=0$}\\
t,& \text{if $\sigma=0$}
\end{cases}\]
and $\gamma$ is a unit speed geodesic of $(M,F)$. Moreover, if $(M,F)$ is geodesically complete and $W$ is a complete homothetic vector field, then $(M,\hat{F})$ is geodesically complete.
\end{thm}
\noindent This theorem generalizes to pseudo-Finsler metrics the central result in \cite{Rob07} and \cite{HuMo11}. It is remarkable that we get the same expression independently if the translation is straight or reverse.
As for the flag curvature, we derive the following theorem
\begin{thm}\label{theoremcurvature}
Given $u\in A$ and a $g_u$-nondegenerate $2$-plane $\Pi$ containing $u$, let $w\in\Pi$ be such that with $g^{\hat{F}}_u(u,w)=0$. Then,
$\tilde{\Pi}:={\rm span}\{u/\hat{F}(u)-W,w\}$ is a $(u/\hat{F}(u)-W)$-nondegenerate (w.r.t. $F$) $2$-plane and
\begin{equation}\nonumber
K_{\hat{F}}\bigl(u\hspace{0.05cm},\hspace{0.05cm}\Pi\bigr)~=~K_{F}\bigl(u/\hat{F}(u)-W\hspace{0.05cm},\hspace{0.05cm}\tilde{\Pi}\bigr)-\frac{1}{4}\sigma^2,
\end{equation}
\end{thm}
\noindent  This is a well known result in the classical case where 
 $F$ is a Finsler metric and $F(-W)<1$ (see \cite{MoHu07} and \cite{BCS04} when $F$ is Riemannian).
 It turns out that last theorem is somewhat natural from the point of view of the theory of linear symplectic invariants of a generic class of curves of Lagrangean subspaces, called {\it fanning curves}, first introduced by Ahdout \cite{Ah}, and then developed by \'Alvarez Paiva and Dur\'an \cite{AD}, and \'Alvarez Paiva, Dur\'an and Vit\'orio \cite{ADH} (see also \cite{Henrique}) inspired by work of Foulon \cite{Foulon} and Grifone \cite{Grifone}.
We have thus considered appropriate to introduce the notion of flag curvature of a pseudo-Finsler metric using the fanning curves approach. This provides a novel way of thinking of flag curvature in contrast to the usual definitions that use the well known Finslerian connections (Berwald, Cartan, Chern). We do that in $\S$5.1, $\S$5.2 and $\S$5.3. Then we show that the relation obtained between the co-geodesic flows of $F$ and $\hat{F}$ directly implies  a linear symplectic equivalence (up to reparametrization) of the corresponding fanning curves (Theorem \ref{theoremcurves}), which in turn gives the Theorem \ref{theoremcurvature}.

We will finish the paper with a section of conclusions and consequences $\S$6. In particular, as a byproduct of our investigations, we extend all the Randers metrics with constant flag curvature to conic Finsler metrics which are geodesically complete and we also provide the natural candidates for a classification of pseudo-Randers-Kropina metrics with constant flag curvature.
\\\\\\\\\\\\\\\\\\\\\\

\section{Pseudo-Finsler metrics and Zermelo navigation}


\subsection{Pseudo-Minkowski norms}
Let us begin by introducing the notion of (conic) pseudo-Minkowski norm. Along this section we will denote by $V$ a vector space of dimension $n$.
\begin{defi}\label{defpseudo}
Let $V$ be a vector space and $A$ a conic open connected subset of $V\setminus \{0\}$, namely,  an open subset such that $\lambda v\in A$ for every $v\in A$ and $\lambda>0$.  A smooth function $F:A\rightarrow(0,+\infty)$ is said to be a (conic) pseudo-Minkowski norm  on $V$ if
\begin{enumerate}[(i)]
\item it is positive homogeneous of degree $1$, namely, $F(\lambda v)=\lambda F(v)$ for every $v\in A$ and $\lambda>0$,
\item for any $v\in A$, the fundamental tensor $g_v$ defined as
\begin{equation}\label{fundtensor}
g_v(u,w)=\frac 12\left. \frac{\partial^2}{\partial t\partial s}F\left(v+ t u+sw\right)^2\right|_{t=s=0}
\end{equation}
for every $u,w\in V$ is non-degenerate.
\end{enumerate}
\end{defi}
We will understand that the pseudo-Minkowski space is conic without saying it explicitly. Recall that from the homogeneity of $F$,
\begin{equation}\label{gv}
\frac{1}{2}DF^2(v)(w)~=~g_v(v,w)
\end{equation}
and
\begin{equation}\label{df^2}
\frac{1}{2}DF^2(v)(v)~=~g_v(v,v)~=~F(v)^2.
\end{equation}
The indicatrix of a conic pseudo-Minkowski space $(V,F)$ is defined as the subset $$\Sigma~=~\{ v\in A: F(v)=1\}.$$ From \eqref{gv} and \eqref{df^2}, and the connectivity of $A$, it follows that $\Sigma$ is a smooth embedded, connected hypersurface of $V$, with
\begin{equation}
T_v\Sigma~=~\{w\in V:g_v(v,w)=0\}
\end{equation}
and with everywhere transverse position vector. Moreover, if $\xi$ denotes the opposite to the position vector, which is transverse  to $\Sigma$, and $\sigma_\xi$ is the second fundamental form of $\Sigma$ w.r.t. $\xi$, a direct computation (see for instance \cite[Theorem 2.14]{JaSan11}) shows that
\begin{equation}\label{sff}
g_v(X,Y)~=~\sigma^\xi(X,Y)
\end{equation}
for $X$ and $Y$ tangent to $\Sigma$, so that $\sigma^\xi$ is non-degenerate. Actually, these properties completely characterize indicatrices:
\begin{prop}\label{charindicatrix}
Let $V$ be a vector space. Then a subset $\Sigma$ of $V$ is the indicatrix of a pseudo-Minkowski norm $F$ if and only if it is a smooth, embedded, connected, hypersurface of $V$, the position vector is transversal to it  and its second fundamental form  with respect to one (and then to all) transversal vector is non-degenerate everywhere.
\end{prop}
\begin{proof}
The proof is straightforward.
\end{proof}

\begin{defi}\label{transMink}
Given a vector space $V$, we say that a pseudo-Minkowski norm $\hat{F}:\hat{A}\rightarrow (0,+\infty)$ is a translation by a vector $w\in V$ of the pseudo-Minkowski norm $F:A\rightarrow (0,+\infty)$ if its indicatrix $\hat{\Sigma}$ is an open connected subset of $\Sigma+w$, where $\Sigma$ is the indicatrix of $F$. In other words, if
\begin{equation}\label{Navrelation}
F\left(\frac{v}{\hat{F}(v)}-w\right)=1,
\end{equation}
for every $v\in \hat{A}$.
\end{defi}
In general, we can get  more than one pseudo-Minkowski norm by translation of a given one.

Note that although the signatures of the second fundamental forms of $\Sigma$ and $\hat{\Sigma}$ coincide if we choose an appropriate transversal vector,
this does not imply that the fundamental tensors of $F$ and $\hat{F}$  can be identified, since the position vector in $\Sigma$, when translated into $\hat{\Sigma}$,  can have opposite orientation to the position vector in $\hat{\Sigma}$ (see Figure \ref{strongTrans2}).

\begin{prop}\label{translatingF}
Given a pseudo-Minkowski norm $F:A\rightarrow (0,+\infty)$ and a vector $w\in V$, the indicatrices of the translations of $F$ by $w$ are the connected
components of $\{v/F(v)+w: F(v)+g_v(w,v)\not=0\}$.  Moreover, let $Z:\hat{A}\rightarrow\R$ be one of such translations, with $g$ the fundamental tensor of $F$, and $h$, the fundamental tensor of $Z$. Then $g_{v/Z(v)-w}(v/Z(v)-w,v)\not=0$ for every $v\in \hat{A}$, and
\begin{enumerate}[(i)]
\item if $g_{v/Z(v)-w}(v/Z(v)-w,v)>0$ for every $v\in\hat{A}$,  the index of $h$ coincides with the index of $g$. In such a case, we say that $Z$ is a {\em straight  translation } of $F$.
\item If $g_{v/Z(v)-w}(v/Z(v)-w,v)<0$ for every $v\in \hat{A}$,   the index of $h$ is $n-\mu-1$, where $\mu\geq 0$ is the index of $g$. In such a case, we say that $Z$ is a {\em reverse  translation } of $F$.
\end{enumerate}
\end{prop}
\begin{proof}
Observe that Proposition \ref{charindicatrix} implies that  the translation of the indicatrix $\Sigma$ of $F$ by $w$ is the indicatrix of a pseudo-Minkowski norm in the points where the position vector is transversal to $\hat{\Sigma}=\Sigma+w$. Moreover, a vector $v+w\in \Sigma+w$ is transversal to $\Sigma+w$ if and only if $v+w$  is not tangent to $\Sigma$ in $v$, which is equivalent to the condition $g_v(v,v+w)\not=0$. Given an arbitrary vector $v\in A$, then $v/F(v)\in\Sigma$, and the above condition becomes $g_{v/F(v)}(v/F(v),v/F(v)+w)\not=0$, which is equivalent to $F(v)+g_v(w,v)\not=0$, by the homogeneity of $g_v$ with respect to $v$ and \eqref{df^2}, as required. The second part is an easy consequence of \eqref{sff}, since it implies that the index is preserved when you consider the second fundamental form with respect to a transversal vector with the same orientation. The condition $g_{v/Z(v)-w}(v/Z(v)-w,v)>0$ guarantees that $v/Z(v)-w$ and $v$ give the same orientation in the indicatrix $\Sigma$ (or $\Sigma+w$) and then $g$ and $h$ have the same index. When $g_{v/Z(v)-w}(v/Z(v)-w,v)<0$, then $v$ and $v/Z(v)-w$ determine opposite orientations for $\Sigma$, and the second fundamental form of the first one has index equal to $n+\mu-1$, being $\mu$ the index of the second fundamental form with respect to $v/Z(v)-w$. This concludes the proof.
\end{proof}
\subsection{Pseudo-Finsler metrics}
Given a connected manifold $M$ and  a conic open connected subset $A$ of $TM\setminus 0$ in the sense that  $A_p=A\cap T_pM$ is conic, connected  and non-empty  for every $p\in M$, we will say that a smooth function $F:A\rightarrow (0,+\infty)$ is a {\it (conic) pseudo-Finsler metric} if the restrictions $F_p=F|_{T_pM\cap A}$ are (conic) pseudo-Minkowski norms for every $p\in M$.
Let us distinguish two interesting families of pseudo-Finsler metrics. We say
that a pseudo-Finsler metric $F:A\rightarrow (0,+\infty)$ is
\begin{enumerate}[(i)]
\item {\it conic Finsler} if the fundamental tensor in \eqref{fundtensor} is positive definite for every $v\in A$,
\item {\it Lorentz-Finsler} if the fundamental tensor in \eqref{fundtensor} has index $n-1$.
\end{enumerate}
\noindent  These are the pseudo-Finsler metrics with strongly convex indicatrices (i.e., indicatrices have definite second fundamental form), a property which ensures they are the only ones whose geodesics (locally) minimize or maximize the length.


\begin{prop}
Given a manifold $M$, a smooth (connected) hypersurface $\Sigma$ of $TM$ determines a pseudo-Finsler metric on $M$ if and only if at every $v\in \Sigma$, it is transversal to $\{\lambda v: \lambda >0\}$, and $\Sigma_p=\Sigma\cap T_pM$ is an embedded (connected) hypersurface such that its second fundamental form with respect to the position vector  is non-degenerate.
\end{prop}
\begin{proof}
 It follows the same lines as \cite[Proposition 2.10]{CJS14}. Let us define $\Psi:(0,+\infty)\times \Sigma\rightarrow TM$, $(t,w)\rightarrow tw$. The transverality condition and the connectedness of $\Sigma$ imply that $\Psi$ is injective and a local diffeomorphism. Then it is a global diffeomorphism into the image, whose inverse can be written as
 $\Psi^{-1}:A\rightarrow (0,+\infty)\times \Sigma$, $v\rightarrow (F(v),v/F(v))$, where $A={\rm Im}(\Psi)$. Clearly, $F$ is a smooth positive homogeneous map and, by Propositon \ref{charindicatrix}, its fundamental tensor is non-degenerate, which  concludes the implication to the left. Conversely, if $F:A\rightarrow (0,+\infty)$ is a pseudo-Finsler metric, then \eqref{df^2} implies that $1$ is a regular value of $F^2:A\rightarrow (0,+\infty)$ and then $\Sigma=(F^2)^{-1}(1)$ is an embedded hypersurface of $TM$  transversal to $\{\lambda v: \lambda>0\}$ in $v\in \Sigma$.   Proposition \ref{charindicatrix} concludes the statement about the second fundamental form of $\Sigma_p$.
\end{proof}
\begin{prop}\label{goodTrans}
If $(M,F)$ is a pseudo-Finsler manifold and $W$ a smooth vector field on $M$, the transport of every $(T_pM,F_p)$ by $W_p$ in every  connected component of   $\hat{A}=\{v\in T_pM:F(v)+g_v(W,v)\not=0\}$   gives a pseudo-Finsler metric $\hat{F}$ determined by \eqref{Navrelation}.
\end{prop}
\begin{proof}
It follows from Proposition \ref{translatingF} and the observation that the map $TM\rightarrow TM$, $v\rightarrow v+W$, is an isomorphism of fiber bundles and then $\Sigma+W$ is transversal to the position vector  in $TM$ if and only if $\Sigma_p+W_p$ is transversal to the position vector in $T_pM$ for every $p\in M$.
\end{proof}
\begin{defi}\label{CharacTrans}
Following Proposition \ref{translatingF}, we will say that the translation of $F$ by a smooth vector field $W$ is {\em straight} (resp. {\em reverse}) if  $F(v)+g_v(v,W)>0$ (resp. $F(v)+g_v(v,W)<0$)  for every $v\in A$.
\end{defi}
\begin{rem}
 Observe that definition of (conic) pseudo-Finsler metric can be generalized in two ways:
\begin{enumerate}
\item we can consider the case in that the fundamental tensor $g$ in \eqref{fundtensor} is degenerate as in \cite{JaSan11}. Along this paper we will need the non-degeneracy of the fundamental tensor in order to define the volume form in \S \ref{Matsumoto} or the Legendre transformation and the flag curvature in \S \ref{legendre} and \S \ref{sectionflagcurvature}.
\item  we also can consider the case in that the pseudo-Finsler metric is a positive homogeneous function $L:A\rightarrow \R$ of degree $2$. This is the most general case in that for example the Chern connection \cite[Remark 2.8]{JavSoa13} or the Legendre transformation can be defined. When $L\not=0$, we can recover a pseudo-Finsler metric as defined here by making $F=\sqrt{L}$, if $L>0$ or $F=\sqrt{-L}$ if $L<0$. However we cannot consider the region in that $L=0$. The reason is that when we translate the indicatrix $L=0$ we do not obtain the indicatrix of a new pseudo-Finsler metric. In fact, at every $p\in M$, the indicatrix ${\mathcal C}_p=\{v\in A_p:L(v)=0\}$ is a degenerate hypersurface of $T_pM$ with  degenerate directions given by rays contained in ${\mathcal C}_p$ which cross the zero vector of $T_pM$. When we translate ${\mathcal C}_p$, this property is lost.
\end{enumerate}
\end{rem}
\subsection{Navigation in semi-Riemannian metrics}

In general, it is not possible to compute explicitly the expression of the translation of a pseudo-Finsler metric, but when we consider a semi-Riemannian metric, the situation is simpler.
\begin{prop}\label{transZer}
Let $g$ be a semi-Riemannian metric in a manifold $M$ and $W$ a smooth vector field in $M$, and consider the conic pseudo-Finsler metric $F(v)=\sqrt{g(v,v)}$ defined in  $A=\{v\in TM:g(v,v)>0\}$. Let $h$ be the  (possibly degenerate) semi-Riemannian metric in $M$ determined by
\[h(v,v)=g(v,W)^2+g(v,v)(1-g(W,W)).\]
Then there are at most two translations of $F$  by a smooth vector field $W$, denoted by $Z_{1}$ and $Z_{-1}$, and given by
\begin{equation}\label{Zermelo1}
Z_\varepsilon(v)=\frac{g(v,W)-\varepsilon \sqrt{h(v,v)}}{g(W,W)-1}
\end{equation}
when $g(W,W)\not=1$ and alternatively by
\begin{equation}\label{Zermelo2}
Z_\varepsilon(v)=\frac{g(v,v)}{g(v,W)+\varepsilon \sqrt{h(v,v)}},
\end{equation}
when $g(v,v)\not=0$. They are defined in the subset $A_\varepsilon$, which is defined as
\[A_\varepsilon=
\{v\in TM: \varepsilon g(v,v)>0\}\cup \{ v\in TM : \varepsilon g(v,W)<0; h(v,v)>0\}\]
if $\varepsilon(1-g(W,W))>0$,
\[A_\varepsilon=\{ v\in TM : \varepsilon g(v,v)>0;\varepsilon g(v,W)>0; h(v,v)>0\}\]
if $\varepsilon(1-g(W,W))<0,$ and
\[A_\varepsilon=\{v\in TM: g(v,v)g(v,W)>0; \varepsilon g(v,W)>0\}\]
if $g(W,W)=1$. Moreover, $Z_1$ (resp. $Z_{-1}$) is a straight (resp. reverse)  translation of $F$.
\end{prop}
\begin{proof}
The Zermelo metric $Z$ associated to $F$ is determined by
\[g\left(\frac{v}{Z(v)}-W,\frac{v}{Z(v)}-W\right)=1,\]
(see \eqref{Navrelation}). Solving a quadratic equation we get that
\begin{equation}\label{ZermeloExp1}
Z_\varepsilon(v)=\frac{g(v,W)-\varepsilon \sqrt{g(v,W)^2+g(v,v)(1-g(W,W))}}{g(W,W)-1}.
\end{equation}
As this expression is not defined when $g(W,W)=1$, we will obtain an alternative expression by conjugation:
\begin{equation}\label{ZermeloExp2}
Z_\varepsilon(v)=\frac{g(v,v)}{g(v,W)+\varepsilon\sqrt{g(v,W)^2+g(v,v)(1-g(W,W))}}.
\end{equation}
Both metrics are defined in the vectors $v$ that make $Z_\varepsilon(v)$ positive. Observing that
\[Z_\varepsilon(v)=\frac{\varepsilon g(v,W)- \sqrt{h(v,v)}}{\varepsilon(g(W,W)-1)}=\frac{\varepsilon g(v,v)}{\varepsilon g(v,W)+ \sqrt{h(v,v)}},\]
we deduce easily that the biggest open subset where it can be defined is $A_\varepsilon$ (namely, distinguish between $\varepsilon g(v,v)<0$ or $\varepsilon g(v,v)\geq 0$ to get all the cases). Now observe that $g(v/Z_\varepsilon(v)-W,v)=\varepsilon \sqrt{h(v,v)}$, which identifies $Z_1$ as a straight translation of $F$ and $Z_{-1}$ as a reverse translation (see Definition \ref{CharacTrans} and Proposition \ref{translatingF}). Finally, observe that in the vectors where $h(v,v)=0$, even if sometimes $Z_\varepsilon(v)$ is positive, the translation is not well-defined because $\Sigma+W$ is not transverse to the position vector.
\end{proof}


Let us consider now Randers and Kropina metrics constructed using a semi-Riemannian metric $h$ and a one-form $\beta$, which in the following will be called respectively, pseudo-Randers and pseudo-Kropina metrics.
Namely, a pseudo-Randers metric is defined as
\begin{equation}\label{pseudo-Randers}
R_\epsilon(v)=\epsilon \sqrt{h(v,v)}+\beta(v)
\end{equation}
for $v\in A=\{v\in TM: h(v,v)>0; \epsilon\sqrt{h(v,v)}+\beta(v)>0\}$ and $\epsilon=-1,1$. Observe that in this definition we include the case in that $h$ is Riemannian but the norm of $\beta$ in $h$ is bigger than one. In such a case, there exists some $v\in TM$ such that $\sqrt{h(v,v)}+\beta(v)<0$, but for these vectors we can consider $R(v)=-\sqrt{h(v,v)}-\beta(v)$.

Moreover, we will define a pseudo-Kropina metric as
\begin{equation}\label{pseudo-Kropina}
K(v)=\frac{h(v,v)}{\beta(v)},
\end{equation}
where $v\in A=\{v\in TM: h(v,v)\beta(v)>0\}$. In both cases, we will denote by $B$ the vector field $h$-equivalent to $\beta$. Moreover, we will say that the pseudo-Randers (resp. pseudo-Kropina) metric is non-degenerate if $h(B,B)\not=1$ (resp. $h(B,B)\not=0$) at every point. It is clear that pseudo-Randers and pseudo-Kropina metrics are positive homogeneous, we will see below that the fundamental tensor is non-degenerate up to some exceptional cases (when $h(B,B)=1$ in the pseudo-Randers case and $h(B,B)=0$ in the pseudo-Kropina case).
\begin{defi}\label{randers-kro-defi}
We say that a pseudo-Finsler metric is pseudo-Randers-Kropina if at every point it can be expressed as \eqref{pseudo-Randers} with $h(B,B)\not=1$ or as in \eqref{pseudo-Kropina} with $h(B,B)\not=0$.
\end{defi}
 Observe that the semi-Riemannian metric $h$ does not to be defined in the whole manifold. Indeed,  the Zermelo metric obtained as the translation of a semi-Riemannian metric (see \eqref{Zermelo1} and \eqref{Zermelo2}) is always pseudo-Randers or pseudo-Kropina at every point, but the
 semi-Riemannian metrics used for the definition are not defined in the whole manifold. Let us see that the converse holds.
In the following we will understand that the indicatrix of a semi-Riemannian metric $g$ in a manifold $M$  is the subset $\Sigma=\{v\in TM: g(v,v)=1\}$.
\begin{prop}\label{Randers-Kro}
The family of pseudo-Randers-Kropina metrics coincides
with the family of Zermelo metrics obtained translating the indicatrix of a semi-Riemannian metric. More precisely:
\begin{enumerate}[(i)]
\item a pseudo-Randers metric as in \eqref{pseudo-Randers} is the translation by the vector $W=-B/\delta$ of the semi-Riemannian metric $g$ determined by
\begin{equation}\label{gdeh}
g(v,v)=\delta (h(v,v)-h(B,v)^2)
\end{equation}
for $v\in TM$, where  $\delta=1-h(B,B)$. Moreover,
\begin{enumerate}[(a)]
\item  if $\epsilon \delta>0$ (resp. $\epsilon \delta<0$) $R_\epsilon$ is a straight (resp. reverse) translation of $g$
\item  if $\epsilon=1$ (resp. $\epsilon=-1$),  the index of the fundamental tensor of $R_1$
(resp. $R_{-1}$) coincides with the index of $h$ (resp. is equal to $n-1-{\rm ind}(h)$).
\end{enumerate}
\item  a pseudo-Kropina metric as in \eqref{pseudo-Kropina} is the translation of the semi-Riemannian metric $g=4 h/ h(B,B)$ by the vector field $W=\frac{1}{2} B$. If $h(B,B)h(v,v)>0$ (resp. $h(B,B)h(v,v)<0$), then $K$ is a straight (resp. reverse) translation and the index of the fundamental tensor of $K$ is equal to the index of $g$ (resp. $n-1-{\rm ind}(g)$).
\end{enumerate}
\end{prop}
\begin{proof}
It is trivial that every Zermelo metric as in \eqref{ZermeloExp2} (or \eqref{ZermeloExp1}) is pseudo-Randers or
pseudo-Kropina. Moreover, if $g(W,W)\not=1$, then the expression in \eqref{ZermeloExp1} is pseudo-Randers with
\[h(v,v)=\frac{1}{|\alpha|^2}g(v,W)^2+\frac{\alpha}{|\alpha|^2}g(v,v)\]
and $\beta(v)=-\frac{1}{\alpha}g(v,W)$, where $\alpha=1-g(W,W)$. Using both relations we deduce that $h(v,W)=-\frac{\alpha}{|\alpha|^2}h(v,B)$ for every $v\in TM$ and then $B=-\frac{|\alpha|^2}{\alpha}W$. It follows in particular that $h(B,B)=g(W,W)$ and $1-h(B,B)=\alpha\not=0$. When $g(W,W)=1$, the expression in \eqref{ZermeloExp2} shows that $Z_\varepsilon$ is pseudo-Kropina with $h=g$ and $\beta(v)=2g(v,W)$, and then $B=2 W$, which implies that $h(B,B)=4 g(W,W)=4\not=0$.

 Consider now a pseudo-Randers metric as in \eqref{pseudo-Randers}  with $h(B,B)\not=1$ and let $g$ be the metric given in \eqref{gdeh}. Then a straightforward computation shows that
\[1-g(W,W)=\delta\not=0,\quad\frac{g(v,W)}{g(W,W)-1}=h(v,B),\]
and
\[g(v,W)^2-g(v,v)(g(W,W)-1)=\delta^2 h(v,v),\]
thus, $R$ is
one of the pseudo-Randers metrics obtained translating the indicatrix of $g$ with $W$ given in
\eqref{ZermeloExp1}. In fact, observe that the sign of $\epsilon \delta$ must coincide with the sign of $\varepsilon$ in 	Proposition \ref{transZer}, which implies the statements about the character of the translation (straight or reverse) in part $(a)$. Moreover, observe that the index of $g$ coincides with the index of $h$ if $\delta>0$ and ${\rm ind}(g)=n-1-{\rm ind}(h)$ if $\delta<0$.  This follows easily from the following three facts (1) $g(B,B)=\delta^2 h(B,B)$, (2) the $g$-orthogonal space to $B$ coincides with the $h$-orthogonal space to $B$ because $g(v,B)=\delta^2 h(v,B)$ and (3) if $v,w$ belong to the $g$-orthogonal space to $B$, $g(v,w)=\delta h(v,w)$. This, together with Proposition \ref{translatingF} and part $(a)$, gives part $(b)$.

For part $(ii)$, the indicatrix of a pseudo-Kropina metric with $h(B,B)\not=0$ is given by the solutions of
\[h(v,v)=h(v,B)\]
or equivalently
\[4h(v-B/2,v-B/2)/h(B,B)=1,\]
which means that $K$ is the translation of $g$ by $W=B/2$. The statements about the character of the translation and the index of the fundamental tensor of $K$ follow easily from Propositions \ref{transZer} and \ref{translatingF}, respectively, since the sign of $\varepsilon$ in Proposition \ref{transZer} coincides with the sign of $h(B,B)h(v,v)$.
\end{proof}

\section{Zermelo navigation and Matsumoto tensor}\label{Matsumoto}

 In this section we give the definition of the Matsumoto tensor of a (conic) pseudo-Finsler metric. 
The main goal will be to prove Theorem \ref{thmmatsumoto}.
 For the sake of completeness we have included a short review of some concepts from affine differential geometry, following \cite{NoSa94}.

\subsection{The Matsumoto Tensor of a (conic) pseudo-Finsler metric}
 Let us first restric our  attention to a single vector space $V$.
The Cartan tensor of a (conic) pseudo-Minkowski norm $F:A\subseteq V\setminus\{0\}\rightarrow (0,+\infty)$ is defined as
\[C_v(X,Y,Z)=\frac 14 \left.\frac{\partial^3}{\partial t_3\partial t_2\partial t_1}
F(v+t_1X+t_2Y+t_3Z)^2\right|_{t_1=t_2=t_3=0}\]
for any $v\in A$ and $X,Y,Z\in V$. The homogeneity of $F$ implies that $C$ is completely symmetric, $C_v(v,\cdot,\cdot)=0$, and $C_{\lambda v}=(1/\lambda)C_v$ if $\lambda>0$. The $g$-metric contraction of the Cartan tensor is called the {\it mean Cartan torsion}, and denoted by $I$; namely, if $b_1,b_2,\ldots,b_n$ is a basis of $V$, and we denote $g_{ij}=g_v(b_i,b_j)$, and $g^{ij}$, the coefficients of the inverse matrix of $\{g_{ij}\}_{i,j=1,\dots,n}$, then
\[I_v(X)=\sum_{i,j=1}^n g^{ij}C_v(X,b_i,b_j)\]
for every $X\in V$.
\\ Let us recall that the {\it angular metric} of $F$ is defined as
\[h_v(X,Y)=g_v(X,Y)-\frac{1}{F(v)^2}g_v(v,X)g_v(v,Y),\]
for every $v\in A$, and $X,Y\in V$. Note that $h_v$ and $g_v$ coincide on $T_v\Sigma$, for $v\in\Sigma$.
\begin{defi}\label{defMat}
Given $v\in A$, the {\it Matsumoto tensor} of $F$ is the symmetric tri-tensor ${\rm M}_v:V\times V\times V \rightarrow \R$ given by
\begin{multline*}
{\rm M}_v(X,Y,Z)=C_v(X,Y,Z)\\
-\frac{\varepsilon}{n+1}\left(I_v(X) h_v(Y,Z)+I_v(Y) h_v(X,Z)+I_v(Z) h_v(X,Y)\right),
\end{multline*}
where $\varepsilon=1$ or $-1$, according to  the index of $g$ is even or odd,  respectively.
\end{defi}
\noindent As ${\rm M}_v(v,\cdot,\cdot)=0$ and ${\rm M}_{\lambda v}=(1/\lambda) {\rm M}_v$, it follows that the tensor ${\rm M}$ contains the same information as its pull-back to
$\Sigma$, which we still denote by ${\rm M}$.
\\\\
\noindent  If now $(M,F)$ is a (conic) pseudo-Finsler manifold, its Matsumoto tensor ${\rm M}$ is defined by just performing the above construction in each tangent space.  


\begin{rem}\label{conicLorentzMat}
 It is interesting to point out that it is possible to describe all the conic Finsler and Lorentz-Finsler metrics with trivial Matsumoto tensor using Theorem \ref{thmmatsumoto} and Propositions \ref{transZer} and \ref{Randers-Kro}:
\begin{enumerate}
\item  In the conic Finsler case, these are those in \eqref{Zermelo2} with $g$ a Riemannian metric and $\varepsilon=1$ or $g$ a Lorentzian metric and $\varepsilon=-1$. Moreover, at every point they can be expressed as in \eqref{pseudo-Randers} with $\epsilon=1$ and $h$ a Riemannian metric, or $\epsilon=-1$ and $h$ a Lorentzian metric, or as in \eqref{pseudo-Kropina} with $h$ Riemannian.

\item  In the Lorentz-Finsler case, these are those in \eqref{Zermelo2} with $g$ a Lorentzian metric and $\varepsilon=1$ or $g$ a Riemannian metric and $\varepsilon=-1$. Moreover, at every point they can be expressed as in \eqref{pseudo-Randers} with $\epsilon=1$ and $h$ a Lorentzian metric, or $\epsilon=-1$ and $h$ a Riemannian metric, or as in \eqref{pseudo-Kropina} with $h$ Lorentzian.
\end{enumerate}
\end{rem}

\subsection{Affine differential geometry of the indicatrix}
 Throughout this section,  let be fixed a vector space $V$, a parallel volume form $\Omega$ in $V$, and the orientation that it induces in $V$.
 \par We now recall some concepts of affine differential geometry for a connected hypersurface $S\subset V$, following \cite{NoSa94}.
 The choice of an everywhere transverse vector field $\xi$ to $S$ induces in $S$ a connection $\nabla^\xi$, a second fundamental form $h_\xi$,
 a volume form $\theta^\xi$, and an orientation via the relations
 $$\nabla_X Y=\nabla_X^\xi Y+h^\xi(X,Y)\xi,~~\theta^\xi(X_1,\cdots,X_{n-1})=\Omega(X_1,\cdots,X_{n-1},\xi),$$
 where $\nabla$ is the canonical connection of $V$, and the orientation in $S$ is the induced by $\theta^\xi$.

\begin{defi}
The pair $(S,\xi)$, or simply $\xi$, is said to be {\it equiaffine} if $\theta^\xi$ is $\nabla^\xi$-parallel; this, in turn, is equivalent to $\nabla_X\xi$ be tangent to $S$ for every $X$ tangent to $S$. If $\xi$ is equiaffine, the tri-tensor field $\mathcal{C}^\xi$ in $S$ given by
$$\mathcal{C}^\xi(X,Y,Z)~=~(\nabla_X^\xi h^\xi)(Y,Z)$$
is completely symmetric and is called the associated {\it cubic form}.
\end{defi}
Under a change
\begin{equation}\label{change}
\overline{\xi}~=~\phi\xi+P,
\end{equation}
where $\phi$ and $P$ are, respectively, a nowhere null smooth function and a vector field on $S$, we have the following:
\begin{lemma}\label{lemmachange}
If $\xi$ is equiaffine, then a change (\ref{change}) produces another equiaffine transverse vector field if, and only if, $h^\xi(P,~\cdot~)=-d\phi$.  In this case, the corresponding cubic forms are related by
$$\phi\mathcal{C}^{\overline{\xi}}(X,Y,Z)=\mathcal{C}^\xi(X,Y,Z)-\frac{1}{\phi}\bigl(X(\phi)h^\xi(Y,Z)+Y(\phi)h^\xi(Z,X)+
Z(\phi)h^\xi(X,Y)\bigr)$$
\end{lemma}
\begin{proof}
An immediate consequence of the formulas in \cite[Proposition 2.5]{NoSa94}.
\end{proof}

Suppose now that $S$ is {\it non-degenerate}, that is, it has everywhere non-degenerate second fundamental form with respect to some (and hence {\it all}) transverse vector field $\xi$. In this case, $h^\xi$ induces a volume form $\omega_{h^\xi}$ in $S$ by
$$\omega_{h^\xi}(X_1,\cdots,X_{n-1})~=~|{\rm det}[h^\xi(X_i,X_j)]|^{1/2}$$
if $X_1,\ldots,X_{n-1}$ is a positive basis of $T_vS$.

\begin{prop}
If $S\subset V$ is a non-degenerate connected hypersurface, there is, up to sign, only one transverse vector field $\xi$ such that
\begin{enumerate}[(i)]
\item $\xi$ is equiaffine;
\item The volume forms $\theta^\xi$ and $\omega_{h^\xi}$ coincide.
\end{enumerate}
Any  of these is called the {\rm Blaschke normal field} of $S\subset V$.
\end{prop}

Indeed, as shown in \cite[page 41]{NoSa94}, starting with any equiaffine transverse vector field $\xi$, the change (\ref{change}) will produce
the Blaschke normal fields if, and only if,
\begin{equation}\label{formula}
|\phi|~=~|{\rm det}_{\theta^\xi}h^\xi|^\frac{1}{n+1}~,~~~h^\xi(P,~\cdot~)~=~-d\phi.
\end{equation}
Here, the value of ${\rm det}_{\theta^\xi}h^\xi$ at $v\in S$ is ${\rm det}[h^\xi(X_i,X_j)]$, where $X_1,\ldots,X_{n-1}\in T_v S$ are any vectors such that
$\theta^\xi(X_1,\ldots,X_{n-1})=1$.
\par We now apply these considerations to the case where $S$ is the indicatrix $\Sigma$ of a (conic) pseudo-Minkowski norm $F$. Let $\xi$ be the opposite of the position vector field. In
this case,
\begin{enumerate}
\item[(I)] $\xi$ is equiaffine.
\item[(II)] $h^\xi=h$, where $h$ is (the pull-back of) the angular metric; this follows from (\ref{sff}).
\item[(III)] The cubic form $\mathcal{C}^\xi$ is twice the (pull-back of) Cartan tensor $C$:
\begin{eqnarray}
2C_v(X,Y,Z) & = &  \left.\frac{\partial}{\partial t}
g_{v+tX}(Y,Z)\right|_{t=0}=(\nabla_Xg)(Y,Z)=(\nabla^\xi_Xh^\xi)(Y,Z)\nonumber\\
& = & \mathcal{C}^\xi_v(X,Y,Z),\nonumber
\end{eqnarray}
where $v\in\Sigma$ and  $X,Y,Z\in T_v\Sigma$.
\end{enumerate}

\subsection{Proof of Theorem \ref{thmmatsumoto}} Let $\omega_g$ be the volume form in $A$ induced by the semi-Riemannian metric $g$, and let $\psi:A\rightarrow\R$ be the function such that $\omega_g=\psi\Omega$. Note that $\psi>0$ as we are considering the orientation in $V$ induced by $\Omega$.
\begin{prop}
The cubic form $\mathcal{C}$ of $\Sigma$ associated to (one of) the Blaschke normal field is equal to $\frac{2}{\phi}{\rm M}$, where ${\rm M}$ is the (pull-back of) Matsumoto tensor of $F$, and $\phi=\psi^\frac{2}{n+1}$.
\end{prop}
\begin{proof}
Due to Lemma \ref{lemmachange} and (I), (II), (III) above, all we have to show is that $\phi$ satisfies (\ref{formula}) and that $X(\psi)/\psi=\varepsilon I_v(X)$ for all
$v\in\Sigma$ and $X\in T_v\Sigma$.
\\ Note that
$$
\theta^\xi(X_1,\cdots,X_{n-1})=\frac{1}{\psi}\omega_g(X_1,\cdots,X_{n-1},\xi)=\frac{1}{\psi}|{\rm det} [h^\xi(X_i,X_j)]|^{1/2}
$$
for $X_1,\ldots,X_{n-1}\in T_v\Sigma$.
This follows from the definitions of $\theta^\xi$ and $\psi$, and the fact that $h^\xi_v=g_v$ in $T_v\Sigma$, $g_v(v,v)=1$ and $v$ is $g_v$-orthogonal to $T_v\Sigma$. Therefore, $\theta^\xi(X_1,\cdots,X_{n-1})=1$ if, and only if, $\psi^2=|{\rm det} [h^\xi(X_i,X_j)]|$, establishing the equality $ \phi^{n+1}=|{\rm det}_{\theta^\xi} h^\xi|$.
\\ Let now $b_1,\ldots,b_n$ be a basis of $V$ with $\Omega(b_1,\ldots,b_n)=1$, and denote, as before, $g_{ij}=g_v(b_i,b_j)$ and $g^{ij}$ the coefficients of the inverse matrix of the $g_{ij}$. Then $\psi=|{\rm det} [g_{ij}]|^{1/2}$ and therefore, for $X\in T_v\Sigma$,
\begin{eqnarray}
 X(\psi)/\psi & = & \frac{1}{2|{\rm det}[g_{ij}]|}X({\rm det}[g_{ij}])=\frac{1}{2|{\rm det}[g_{ij}]|}{\rm det}[g_{ij}]\sum_{k,l}X(g_{kl})g^{kl}\nonumber\\
 & = & \frac{{\rm det}[g_{ij}]}{|{\rm det}[g_{ij}]|}\sum_{k,l}C_v(X,b_k,b_l)g^{kl}\nonumber
 \end{eqnarray}
which in turn is equal to $\varepsilon I_v(X)$.
\end{proof}
Theorem \ref{thmmatsumoto} now follows from Proposition \ref{Randers-Kro}  and the following classic result (see \cite[Theorem 4.5]{NoSa94}).
\begin{thm}[Maschke, Pick, Berwald]
If dim $V\geq 3$, then the cubic form associated to the Blaschke normal field of a non-degenerate connected hypersurface $S\subset V$ vanishes if, and only if, $S$ is a hyperquadric.
\end{thm}


\section{The Legendre dual of a conic pseudo-Finsler metric}\label{legendre}
In this section we will discuss the Legendre transformation of a conic pseudo-Finsler metric and then we will apply it to give a nice dual description
 of the process of translating such a metric.\\\\
Throughout this paper, $\tau:T^\ast M\rightarrow M$ will denote the projection map.
\subsection{Definitions and general results}
The {\it Legendre transformation} of a conic pseudo-Finsler metric $(M,F)$ is the map
\begin{equation}\nonumber
\mathscr{L}_F~:~A\longrightarrow T^\ast M,~~~~~~\mathscr{L}_F(u)=\frac{1}{2}D_fF^2(u),
\end{equation}
where $D_f$ stands for the fiber-derivative. Then, $\mathscr{L}_F$ is positively homogeneous of degree one, and
\begin{equation}\label{legendretransform}
\mathscr{L}_F(u)~=~g_u(u,\cdot),
\end{equation}
so that $\mathscr{L}_F(u)$ is the only covector characterized by
\begin{equation}\label{affinetangenthyperplane}
\mathscr{L}_F(u)\Bigl(u/F(u)+T_{u/F(u)}\Sigma_{\pi(u)}\Bigr)~=~F(u).
\end{equation}
The fiber-derivative of $\mathscr{L}_F$ is given by
\begin{equation}\label{DfL}
D_f\mathscr{L}_F(u)(v)~=~g_u(v,\cdot),
\end{equation}
hence, as the fundamental tensor $g$ is non-degenerate everywhere, it follows that $\mathscr{L}_F$ is a local diffeomorphism.
\begin{defi}
Given $\xi\in\mathscr{L}_F(A)$, we define, for each $v\in A$ with $\mathscr{L}_F(v)=\xi$, the {\it Legendre dual} of $F$ around $\xi$ as the map
\begin{equation}\nonumber
F^\ast~=~F\circ\mathscr{L}_F^{-1}
\end{equation}
defined on some neighborhood of $\xi$ in $\mathscr{L}_F(A)$ onto which $\mathscr{L}_F$ maps diffeomorphically a neighborhood of $v$ in $A$.
\end{defi}

\noindent Viewing $\mathscr{L}_F^{-1}$ and $F^\ast$ as global multivalued functions on $\mathscr{L}_F(A)$, we can describe $F^*$ as follows.

\begin{prop}\label{criticalvalues}
Given $\xi\in\mathscr{L}_F(A)$, we have
\begin{equation}
F^\ast(\xi)  =  \text{Positive critical values of $\xi|_{\Sigma_{\tau(\xi)}}$,}
\end{equation}
and to each corresponding critical point $v\in\Sigma_{\tau(\xi)}$, there is a unique $\lambda>0$ such that $\mathscr{L}_F(\lambda v)=\xi$.  Moreover,
\[\mathscr{L}_F(A)=\{\xi\in T^*M: \xi|_{\Sigma_{\tau(\xi)}} \text{admits positive critical points}\}.\]
\end{prop}
\begin{proof}
Let $\xi=\mathscr{L}_F(u)$, so that $F(u)\in F^\ast(\xi)$. It is straightforward to check that \eqref{affinetangenthyperplane} implies that $u/F(u)$ is a critical point of $\xi|_{\Sigma_{\tau(\xi)}}$ and $\xi(u/F(u))=F(u)>0$. Conversely, if $v\in \Sigma_{\tau(\xi)}$ is a critical point of  $\xi|_{\Sigma_{\tau(\xi)}}$, then $\xi$ is constant on the hyperplane $P=v+T_{v}\Sigma_{\pi(v)}$. If $\xi(v)=\lambda>0$, it follows that $\xi=\mathscr{L}_F(\lambda v)$ (since both coincide in the hyperplane $P$) and $F^*(\xi)=F(\lambda v)=\lambda$.  The last assertion follows easily from the previous reasoning.
\end{proof}

\begin{rem}\label{propminmax}
  If we particularize to the conic Finsler or Lorentz-Finsler case, we get the following description.
If $F$ is $(1)$ conic Finsler, or $(2)$ Lorentz Finsler, and $A$ is convex (or $A=TM\setminus\{\mbox{zero section}\}$ if $F$ is Finsler), then the map
$\mathscr{L}_F$ is a diffeomorphism onto its image and, hence, the Legendre dual $F^\ast$ is globally defined. Moreover,
\begin{eqnarray}
(1)~~F^\ast(\xi) & = & \max_{\Sigma_{\tau(\xi)}}\xi\nonumber\\
(2)~~F^\ast(\xi) & = & \min_{\Sigma_{\tau(\xi)}}\xi.\nonumber
\end{eqnarray}
\end{rem}

For future reference, we note that the inverse of $\mathscr{L}_F$ is simply the {\it Legendre transformation of $F^\ast$}, as can be shown by a direct computation employing the homogeneity of $F$ and $F^\ast$.
\begin{prop}\label{propinverse}
The inverse $\mathscr{L}_F^{-1}:\mathscr{L}_F(A)\rightarrow A$ is the fiber derivative of $(1/2)(F^\ast)^2$:
\begin{equation}\nonumber
\mathscr{L}_F^{-1}(\xi)~=~(1/2)D_f(F^\ast)^2(\xi).
\end{equation}
\end{prop}

\subsection{The Legendre dual of a translation}
Let $(V,\hat{F})$ be a conic pseudo-Minkowski norm obtained as the translation of a  conic pseudo-Minkowski norm $(V,F)$ and a vector  $W$ (recall Definition \ref{transMink}). It will be convenient
to consider the {\it reverse metric} of $F$, which is defined by
\begin{equation}\nonumber
\tilde{F}~:~-A\longrightarrow(0,\infty),~~~\tilde{F}(v)=F(-v).
\end{equation}

\noindent Let us denote by $\hat\Sigma$ and $\Sigma$ the indicatrices of $\hat F$ and $F$, respectively. Observe that the indicatrices of $\tilde{F}$ is $-\Sigma$.
\begin{lemma}\label{charStRv}
$\hat F$ is a straight (resp. reverse) translation of $F$ if and only if $\mathscr{L}_{\hat F}(v)(v-W)>0$ (resp. $\mathscr{L}_{\hat F}(v)(v-W)<0$) for every $v\in \hat \Sigma$.
\end{lemma}
\begin{proof}
It follows from a similar reasoning to that of the proof of Proposition \ref{translatingF}.
\end{proof}
Given $\xi\in\mathscr{L}_{\hat{F}}(\hat{A})$, we have that if a vector $v\in\hat \Sigma$ is a critical point of $\xi|_{\hat\Sigma}$, with positive critical value, then the vector $v-W\in\Sigma$ (resp. $W-v\in-\Sigma$) is a critical point of $\xi|_{\Sigma}$ (resp. $\xi|_{-\Sigma}$) with positive critical value if $\hat F$ is a straight (resp. reverse) translation (recall Proposition \ref{translatingF}). Therefore, Proposition \ref{criticalvalues}  immediately gives the following result.
\begin{prop}\label{propdichotomy}
If $\hat F$ is a straight  (resp. reverse) translation of $F$, then $\mathscr{L}_{\hat F}(\hat{A})\subset\mathscr{L}_{F}(A)$ (resp.
$\mathscr{L}_{\hat F}(\hat{A})\subset\mathscr{L}_{\tilde{F}}(-A)$) and, to each local inverse $\mathscr{L}_{\hat{F}}^{-1}:\mathcal{O}\subset\mathscr{L}_{\hat{F}}(\hat{A})\rightarrow \hat{A}$ corresponds
a unique local inverse $\mathscr{L}_{F}^{-1}:\mathcal{O}\rightarrow A$ (resp. $\mathscr{L}_{\tilde{F}}^{-1}:\mathcal{O}\rightarrow-A$) such that, for $\xi\in\mathcal{O}$,
\begin{eqnarray}
\mathscr{L}_{\hat{F}}^{-1}(\xi)/\hat{F}^\ast(\xi) & = & W+\mathscr{L}_{F}^{-1}(\xi)/F^\ast(\xi)\nonumber\\
(\mbox resp.)~~\mathscr{L}_{\hat{F}}^{-1}(\xi)/\hat{F}^\ast(\xi) & = & W-\mathscr{L}_{\tilde{F}}^{-1}(\xi)/\tilde{F}^\ast(\xi).\label{legendrerelations}
\end{eqnarray}
Furthermore, on $\mathcal{O}$,
\begin{eqnarray}
\hat F^\ast(\xi) & = & F^\ast(\xi)+W(\xi)\nonumber\\
(\mbox{resp.})~~\hat F^\ast(\xi) & = & -\tilde{F}^\ast(\xi)+W(\xi)\label{legendreduallorentz},
\end{eqnarray}
where $W(\xi)$ is the usual action of a covector on a vector, namely $\xi(W)$.
\end{prop}

\section{Flag curvature  and geodesics  of translations}\label{sectionflagcurvature}

Let $(M,\hat{F})$ be a translation of a pseudo-Finsler manifold $(M,F)$ by a vector field $W$. In this section we will be concerned with the  special case in that $W$ is a homothetic vector field (see \eqref{homothetic}),  in which case we will prove Theorems \ref{geodesicflow} and \ref{theoremcurvature}. Let us point out the following corollary of Theorem \ref{theoremcurvature}:

\begin{cor}
If $F$ has constant flag curvature equal to $K$, then $\hat{F}$ has constant flag curvature equal to $K-(1/4)\sigma^2$.
\end{cor}

We start by explaining the new approach  to flag curvature  in $\S$\ref{sectioncurvature}, $\S$\ref{sectioncurves}, $\S$\ref{sectionflagdefi}, and postpone the proofs of Theorems \ref{geodesicflow} and \ref{theoremcurvature}
to $\S$\ref{sectionproof}.

\subsection{Contact geometry of a pseudo-Finsler metric}\label{sectioncurvature}
Throughout this section, let $F$ be a conic pseudo-Finsler metric with domain $A$. As we will only be concerned with local questions, we will therefore suppose that $F$ has a globally defined Legendre dual
\begin{equation}\label{hamiltonian}
\frac{1}{2}(F^\ast)^2~:~\mathscr{L}_F(A)\longrightarrow\mathbb{R}.
\end{equation}
Let us recall that $T^\ast M$, and hence $\mathscr{L}_F(A)$, possesses a canonical symplectic structure $\omega$: this is the differential of the {\it canonical $1$-form} $\alpha$ on $T^\ast M$, which is defined by
\begin{equation}
\alpha(X)~=~\xi\bigl(D\tau(\xi)(X)\bigr)~,~~~~\mbox{for $X\in T_\xi(T^\ast M)$.}
\end{equation}
Viewing (\ref{hamiltonian}) as a Hamiltonian function, it defines a Hamiltonian vector field $S$ on $\mathscr{L}_F(A)$ through the relation
\begin{equation}\label{hamiltonianvectorfield}
\frac{1}{2}D(F^\ast)^2~=~\omega(\cdot\hspace{0.05cm},\hspace{0.05cm}S),
\end{equation}
and the corresponding flow $\psi_t^S$, called co-geodesic flow of $F$, acts on $\mathscr{L}_F(A)$ by symplectic diffeomorphisms. Indeed, it is not hard to show that
\begin{equation}\label{pullbackalpha}
(\psi_t^S)^\ast\alpha~=~\alpha+t\frac{1}{2}D(F^\ast)^2.
\end{equation}
As, by definition, the geodesics of $F$ are the critical points of the energy functional
\begin{equation}\nonumber
E~:~C^F_{p,q}\rightarrow\mathbb{R},~~~E(\gamma)~=~\frac{1}{2}\int_a^bF(\dot{\gamma}(t))^2dt,
\end{equation}
where $C^F_{p,q}$ is the set of piecewise smooth admissible curves joining $p$ to $q$, the Hamilton's principle applies to show that the geodesics of $F$ are precisely the projections onto $M$ of the integral lines of $\psi_t^S$.

The manifold $\mathscr{L}_F(A)$ is foliated by the hypersurfaces $(F^\ast)^{-1}(r)$, $r>0$, each of which is invariant by the flow $\psi_t^S$ as it follows from (\ref{hamiltonianvectorfield}). It will be more convenient to work with the restriction of $\psi_t^S$ to theses hypersurfaces. First, le us recall the following fundamental notion.

\begin{defi}
An {\it exact contact manifold} is a manifold $X$ endowed with a 1-form $\alpha$ with the property that the restriction of $\omega=d\alpha$ to the distribution
$X\ni x\mapsto\mathcal{C}_x:=\ker(\alpha_x)\subset T_xX$ is non-degenerate (hence, symplectic). We call $\mathcal{C}$ the {\it contact plane distribution}.
Furthermore,
\begin{enumerate}
\item an {\it exact contact flow} on $X$ is a flow $\psi_t$ such that $(\psi_t)^\ast\alpha=\alpha$. It follows that such a flow leaves the distribution $\mathcal{C}$ invariant and its derivatives $D\psi_t(x)|_{\mathcal{C}_x}:\mathcal{C}_x\rightarrow\mathcal{C}_{\psi_t(x)}$ will consist of symplectic linear maps.
\item A distribution $\mathcal{L}$ on $X$ is called {\it Legendrean} if, for each $x$, $\mathcal{L}_x$ is a Lagrangian subspace of the symplectic space $(\mathcal{C}_x,\omega_x|_{\mathcal{C}_x})$, i.e., $\omega_x|_{\mathcal{L}_x}=0$ and dim$\mathcal{L}_x=\frac{1}{2}{\rm dim}\mathcal{C}_x$.
\end{enumerate}
\end{defi}

The following result is well-known; we include a proof here as we will need some facts established along it.

\begin{prop}\label{propcontact}
Endowed with the pull-back of the canonical $1$-form $\alpha$, $(F^\ast)^{-1}(r)$ is an exact contact manifold, with contact plane distribution $\mathcal{C}$. Furthermore,
\begin{enumerate}[(i)]
\item The flow $\psi_t^S$ acts on $(F^\ast)^{-1}(r)$ as an exact contact flow.
\item the distribution $\mathcal{L}$ on $(F^\ast)^{-1}(r)$ given by $\mathcal{L}_\xi=T_\xi\bigl(\tau^{-1}(x)\cap(F^\ast)^{-1}(r)\bigr)$ is Legendrean.
\end{enumerate}
\end{prop}

\begin{proof}
Let $C$ be the tautological vector field on $T^\ast M$, i.e., for each $\xi\in T^\ast M$, $C_\xi=\xi\in T_\xi(\tau^{-1}(\tau(\xi)))\subset T_\xi(T^\ast M)$ (where we have identified $T_\xi(\tau^{-1}(\tau(\xi)))$
with $T^\ast_{\tau(\xi)}M$ in the canonical way). It is known that $C$ is the $\omega$-dual of $\alpha$:
\begin{equation}\label{alphaomega}
\alpha~=~\omega(C\hspace{0.05cm},\hspace{0.05cm}\cdot).
\end{equation}
The relations (\ref{hamiltonianvectorfield}) and (\ref{alphaomega}), together with the Euler relation for 2-homogeneous functions, give
$\alpha_\xi(S_\xi)=\frac{1}{2}D(F^\ast)^2(\xi)(C_\xi)=F^\ast(\xi)^2\neq 0$,
showing that $C$ is everywhere transverse to $(F^\ast)^{-1}(r)$ and $S$ is never tangent to the distribution $\mathcal{C}$. Hence,
\begin{equation}\label{decomposition}
T_\xi( T^\ast M)~=~\mathcal{C}_\xi\oplus{\rm span}\{S_\xi,C_\xi\}.
\end{equation}
Again from (\ref{hamiltonianvectorfield}) and (\ref{alphaomega}),
\begin{equation}\label{contactplane}
\mathcal{C}_\xi~=~\ker\omega_\xi(S_\xi\hspace{0.05cm},\hspace{0.05cm}\cdot)\cap\ker\omega_\xi(C_\xi\hspace{0.05cm},\hspace{0.05cm}\cdot)
\end{equation}
and so the decomposition (\ref{decomposition}) is $\omega_\xi$-orthogonal. As $\omega_\xi$ is non-degenerate on $T_\xi(T^\ast M)$ and on ${\rm span}\{S_\xi,C_\xi\}$
(because $\omega_\xi(C_\xi,S_\xi)=\alpha_\xi(S_\xi)\neq 0$), it follows from (\ref{decomposition}) that it is non-degenerate on $\mathcal{C}_\xi$. For assertion $(i)$,
we pull-back the relation (\ref{pullbackalpha}) to $(F^\ast)^{-1}(r)$, to get $(\psi_t^S)^\ast\alpha=\alpha$ (omitting pull-back's).
For $(ii)$, as $\alpha$ vanishes on vectors tangent to $\tau^{-1}(x)$, its pull-backs to the leaves of the foliation $x\mapsto\tau^{-1}(x)\cap(F^\ast)^{-1}(r)$ of $(F^\ast)^{-1}(r)$
are null, and hence $d\alpha$ pulls back to the null form on each leaf.
\end{proof}

\noindent It follows from the last proposition that the flow $\psi_t^S$ carries $\mathcal{L}$ into Lagrangian subspaces of $\mathcal{C}$. Let $\xi\in(F^\ast)^{-1}(r)$ be fixed, and
denote by $\mathbb{V}$ the symplectic vector space $(\mathcal{C}_\xi,\omega|_{\mathcal{C}_\xi})$. The collection of all Lagrangian subspaces of $\mathbb{V}$, denoted by
$\Lambda(\mathbb{V})$, possesses a natural smooth structure and is called the {\it Lagrangian Grasmannian manifold} of $\mathbb{V}$.
\begin{defi}\label{Jacobicurve}
The {\it Jacobi curve} of $(\ref{hamiltonian})$, based at $\xi$, is the smooth curve in $\Lambda(\mathbb{V})$ given by
\begin{equation}
\ell_\xi(t)~=~D\psi_{-t}^S(\mathcal{L}_{\psi_t^S(\xi)}).
\end{equation}
\end{defi}

\subsection{The geometry of curves in $\Lambda(\mathbb{V})$}\label{sectioncurves}
Here, we briefly recall the main concepts of \cite{AD} which will appear in the sequel.
\\ For each $\ell\in\Lambda(\mathbb{V})$, the tangent space $T_\ell\Lambda(\mathbb{V})$ is canonically isomorphic to the space ${\rm Sym}(\ell)$ of symmetric bilinear forms on $\ell$.
Indeed, if $\ell(t)$ is a smooth curve in $\Lambda(\mathbb{V})$ then, for each $\tau$, the following rule defines a symmetric bilinear form $W(\tau)$ on $\ell(\tau)$,
\begin{equation}\nonumber
W(\tau)(a,b)~=~\omega\bigl(\dot{a}(\tau),b\bigr),
\end{equation}
where $a(\cdot):(\tau-\varepsilon,\tau+\varepsilon)\rightarrow\mathbb{V}$ is a smooth curve satisfying $a(\tau)=a$, and $a(t)\in\ell(t)$ for all $t$. We call
the assignment $t\mapsto W(t)$ the {\it Wronskian} of the curve $\ell(t)$.

\begin{defi}
We call the curve $\ell(t)$ {\it fanning} if its Wronskian $W(t)$ is non-degenerate for all $t$.
\end{defi}
\noindent The condition of being fanning may be translated in terms of frames as follows: if $\mathcal{A}(t)=\{a_1(t),...,a_n(t)\}$ is a smooth frame for $\ell(t)$, then $\ell(t)$ is fanning if, and only if,
\begin{equation}\nonumber
\{\mathcal{A}(t),\dot{\mathcal{A}}(t)\}=\{a_1(t),...,a_n(t),\dot{a}_1(t),...,\dot{a}_n(t)\}
\end{equation}
is a smooth frame for $\mathbb{V}$. Hence, if $\mathcal{A}(t)$ is a smooth frame for a fanning curve $\ell(t)$, we may define a smooth curve of endomorphisms ${\bf F}(t):\mathbb{V}\rightarrow\mathbb{V}$ by
\begin{eqnarray}
{\bf F}(t)a_i(t) & = & 0~~~~~\mbox{for all $i$}\nonumber\\
{\bf F}(t)\dot{a}_i(t) & = & a_i(t)~~~~\mbox{for all $i$}\nonumber
\end{eqnarray}

\noindent It turns out that the endomorphisms ${\bf F}(t)$ do not depend on the particular choice of the frame $\mathcal{A}(t)$, but only on $\ell(t)$. We call the assignment $t\mapsto{\bf F}(t)$ the {\it fundamental endomorphism} of $\ell(t)$. Taking the second derivative of ${\bf F}(t)$, we get the main linear invariant of a fanning curve.
\begin{defi}
The {\it Jacobi endomorphism} of $\ell(t)$ is the smooth curve of endomorphisms ${\bf K}(t):\mathbb{V}\rightarrow\mathbb{V}$ defined by
\begin{equation}\nonumber
{\bf K}(t)~=~(1/4)\ddot{{\bf F}}(t)^2.
\end{equation}
For each $t$, ${\bf K}(t)$ restricts to a $W(t)$-symmetric endomorphism of $\ell(t)$. We have the following basic transformation's rule, which will be needed for the proof of Theorem \ref{theoremcurvature}; its proof follows directly from \cite[Theorem 1.4]{AD}.
\begin{prop}\label{propfanning}
If two fanning curves $\ell(t)$ and $\tilde{\ell}(t)$ are related by $\tilde{\ell}(t)={\rm T}\ell(s(t))$, where ${\rm T}:\mathbb{V}\rightarrow\tilde{\mathbb{V}}$ is a linear symplectic (or anti-symplectic) map, and $s$ is a diffeomorphism of $\mathbb{R}$, then, whenever $W(s(t))(a,a)\neq 0$,
\begin{equation}\nonumber
\frac{\tilde{W}(t)\bigl(\tilde{\bf K}(t){\rm T}a\hspace{0.05cm},\hspace{0.05cm}{\rm T}a\bigr)}{\tilde{W}(t)\bigl({\rm T}a\hspace{0.05cm},\hspace{0.05cm}{\rm T}a\bigr)}~=~\dot{s}(t)^2\frac{W(s(t))\bigl({\bf K}(s(t))a\hspace{0.05cm},\hspace{0.05cm}a\bigr)}{W(s(t))\bigl(a\hspace{0.05cm},\hspace{0.05cm}a\bigr)}+\frac{1}{2}\{s(t),t\},
\end{equation}
where $W(t)$, $\tilde{W}(t)$, ${\bf K}(t)$, $\tilde{{\bf K}}(t)$ are the Wronskians and Jacobi endomorphisms of $\ell(t)$ and $\tilde{\ell}(t)$, respectively, and $\{s(t),t\}=(d/dt)(\ddot{s}/\dot{s})-(1/2)(\ddot{s}/\dot{s})^2$ is the {\rm Schwarzian derivative} of $s$.
\end{prop}
\end{defi}

\subsection{Definition of flag curvature via fanning curves}\label{sectionflagdefi}
We go back now to the Jacobi curve $\ell_\xi(t)$; see Definition \ref{Jacobicurve}. The Wronskian $W_\xi(t)$ of $\ell_\xi(t)$ corresponds to the fundamental tensor $g$ of $F$ as we now explain. \\\\
Let $v\in A$ be such that $\mathscr{L}_F(v)=\xi$, and let $\gamma(t)$ be the geodesic of $F$ with $\dot{\gamma}(0)=v$, so that we have $\mathscr{L}_F(\dot{\gamma}(t))=\psi_t^S(\xi)$.
We will now describe an isomorphism
\begin{equation}\label{isomorphismcurve}
\iota_{\dot{\gamma}(t)}~:~\ker g_{\dot{\gamma}(t)}(\dot{\gamma}(t)\hspace{0.05cm},\hspace{0.05cm}\cdot)\longrightarrow\ell_\xi(t).
\end{equation}
For this, note that $\mathcal{L}_{\psi_t(\xi)}$ is canonically a subspace of $T^\ast_{\gamma(t)}M$ and, via the isomorphism $D_f\mathscr{L}_F(\dot{\gamma}(t)):T_{\gamma(t)}M\rightarrow T_{\gamma(t)}^\ast M$, it corresponds to the tangent space, at $\dot{\gamma}(t)$, of the dilatation by $r$ of $\Sigma_{\gamma(t)}$, which in turn is equal to $\ker g_{\dot{\gamma}(t)}(\dot{\gamma}(t)\hspace{0.05cm},\hspace{0.05cm}\cdot)$. So we define $\iota_{\dot{\gamma}(t)}$ by composing the restriction of $D_f\mathscr{L}_F(\dot{\gamma}(t))$ to $\ker g_{\dot{\gamma}(t)}(\dot{\gamma}(t)\hspace{0.05cm},\hspace{0.05cm}\cdot)$ with $D\psi_{-t}^S$. For $t=0$, this is simply
\begin{eqnarray}\label{isomorphismzero}
\iota_v=D_f\mathscr{L}_F(v)|_{\ker g_v(v,\cdot)}:~\ker g_v(v\hspace{0.05cm},\hspace{0.05cm}\cdot) & \longrightarrow &\ell_\xi(0)=\mathcal{L}_\xi\\
 w & \longmapsto & g_v(w,\cdot)\nonumber.
\end{eqnarray}
We refer to \cite{ADH} (see also \cite{Henrique}) for a proof of the following result.
\begin{prop}\label{propwronskian}
Via the isomorphism $(\ref{isomorphismcurve})$, the Wronskian $W_\xi(t)$ of $\ell_\xi(t)$ corresponds to the restriction of $g_{\dot{\gamma}(t)}$ to $\ker g_{\dot{\gamma}(t)}(\dot{\gamma}(t)\hspace{0.05cm},\hspace{0.05cm}\cdot)$.
\end{prop}
\noindent It follows from this proposition that $\ell_\xi(t)$ is a fanning curve; see $\S$\ref{sectioncurves}. Let ${\bf K}_\xi(t)$ be its Jacobi endomorphism.

\begin{defi}\label{defincurvature}
Let $\Pi\subset T_xM$ be a 2-plane containing $v$.
\begin{enumerate}
\item We call $\Pi$ $v$-{\it nondegenerate} if the restriction of $g_v$ to $\Pi$ is nondegenerate. This is the same as requiring that
the quantity $g_v(v,v)g_v(u,u)-g_v(u,v)^2$ be non null for all $u$ such that $\Pi={\rm span}\{u,v\}$.
\item If $\Pi$ is $v$-nondegenerate, we define the {\it flag curvature} of the flag $(v,\Pi)$ as follows: pick a non null vector $w\in\Pi$, with $g_v(v,w)=0$, and let $a\in\ell_\xi(0)$ corresponding to $w$ via $(\ref{isomorphismzero})$. From Proposition \ref{propwronskian} and the $v$-nondegenerescence of $\Pi$, $W_\xi(0)(a,a)\neq 0$, and we may define
\begin{equation}\nonumber
K_F(v,\Pi)~=~\frac{1}{F^2(v)}\frac{W_\xi(0)\bigl({\bf K}_\xi(0)a\hspace{0.01cm},\hspace{0.01cm}a\bigr)}{W_\xi(0)(a\hspace{0.01cm},\hspace{0.01cm}a)}.
\end{equation}
\end{enumerate}
\end{defi}

\noindent Of course, the above definition of flag curvature coincides with the usual ones; a proof of this fact may be found in \cite{ADH} (see also \cite{Henrique}).

 We end this section
by remarking the following simple fact which we will need for the proof of Theorem \ref{theoremcurvature}: the flag curvatures of a metric $F$ and of its reverse $\tilde{F}$ are related by
\begin{equation}\label{remarkreverse2}
K_{\tilde{F}}(v,\Pi)~=~K_F(-v,\Pi).
\end{equation}


\subsection{Proof of Theorems \ref{theoremcurvature} and \ref{geodesicflow}}\label{sectionproof}
Along the proof we will use the following convention for $\epsilon$ and $F_\epsilon$:
\begin{enumerate}[(a)]
\item if $\hat F$ is a straight translation of $F$, then  $\epsilon=1$ and $F_\epsilon=F_1=F$,
\item if $\hat{F}$ is a reverse translation of $F$, then $\epsilon=-1$ and $F_\epsilon=F_{-1}=\tilde{F}$.
\end{enumerate}
Set $\xi=\mathscr{L}_{\hat{F}}(u)$, and let $\mathcal{O}\subset\mathscr{L}_{\hat{F}}(\hat{A})$ be a neighborhood of $\xi$ where $\mathscr{L}_{\hat{F}}$ has a local inverse
\begin{equation}\nonumber
\mathscr{L}_{\hat{F}}^{-1}~:~\mathcal{O}\longrightarrow \hat A.
\end{equation}
According to Proposition \ref{propdichotomy}, this determines a local inverse $\mathscr{L}_{F_\epsilon}^{-1}:\mathcal{O}\rightarrow \epsilon A$ such that
\begin{equation}
\mathscr{L}_{F_\epsilon}^{-1}(\xi)/F_\epsilon^\ast(\xi)~=~-\epsilon W+\epsilon u/\hat{F}(u)
\end{equation}
and, for simplicity, we will denote by $H_\epsilon$ and $H_2$, respectively, the Hamiltonian functions
\begin{equation}\nonumber
{F_\epsilon}^\ast,~\hat{F}^\ast~:~\mathcal{O}\longrightarrow\mathbb{R}.
\end{equation}
For $i=\epsilon, 2$, let $S_i$ and $\hat{S}_i$ be the Hamiltonians vector fields of $H_i$ and $(1/2)H_i^2$, respectively. Then,
$\hat{S}_i=H_iS_i$ and, therefore, for each $r>0$,
\begin{equation}\label{restricflow}
\psi_t^{\hat{S}_i}\mid_{H_i^{-1}(r)}~=~\psi_{rt}^{S_i}\mid_{H_i^{-1}(r)}~~~~,~~i=\epsilon,2.
\end{equation}
Denoting by $H_0$ the Hamiltonian function
\begin{equation}\nonumber
W~:~\mathcal{O}\longrightarrow\mathbb{R},~~~W(\eta)=\eta(W)
\end{equation}
and by $S_0$ its Hamiltonian vector field, it follows from (\ref{legendreduallorentz}) that $H_2=\epsilon H_\epsilon+H_0$ and, hence,
\begin{equation}\label{sumvectorfields}
S_2~=~\epsilon S_\epsilon+S_0.
\end{equation}
Also, it is well known (see, for instance, \cite{Ziller}) that the flow of $S_0$ is simply
\begin{equation}\label{flowS1}
\psi_t^{S_0}=(D\psi_t^W)^\ast~:~T^\ast M\longrightarrow T^\ast M.
\end{equation}
For $i=\epsilon,2$, let $\mathcal{C}^{(i)}$ be the contact plane distribution associated to $(1/2)H_i^2$, and let
\begin{equation}\nonumber
\ell_\xi^{(i)}(t)\in\Lambda\bigl(\mathcal{C}^{(i)}_\xi\bigr)
\end{equation}
be the Jacobi curve of $(1/2)H_i^2$ based at $\xi$. It will be convenient to introduce the following curves on $\Lambda(T_\xi T^\ast M)$,
\begin{equation}\nonumber
\tilde{\ell}_\xi^{(i)}(t)~=~D\psi_{-t}^{S_i}\bigl(\mathcal{V}_{\psi_t^{S_i}(\xi)}T^\ast M\bigr)~~~,~~i=\epsilon,2.
\end{equation}
where $\mathcal{V}T^\ast M$ is the distribution on $T^\ast M$ given by the tangent spaces to the fibers of $\tau:T^\ast M\rightarrow M$. These curves are related to the former ones by

\begin{lemma}
For $i=2,\epsilon$, we have
\begin{equation}\label{eq3}
\ell_\xi^{(i)}(t/H_i(\xi))=\tilde{\ell}_\xi^{(i)}(t)\cap\mathcal{C}^{(i)}_\xi.
\end{equation}
\end{lemma}
\begin{proof}
First note that $\mathcal{V}T^\ast M=\mathcal{L}^{(i)}\oplus{\rm span}\{C\}$. Also, as $H_i$ is positively homogeneous of degree 1, we have $[C,S_i]=0$, and hence $(\psi_t^{S_i})^\ast C=C$. Therefore, it follows from $(\ref{restricflow})$ that $\tilde{\ell}_\xi^{(i)}(t)=\ell_\xi^{(i)}\bigl(t/H_i(\xi)\bigr)\oplus{\rm span}\{C_\xi\}$ and, hence,
$\ell_\xi^{(i)}\bigl(t/H_i(\xi)\bigr)=\tilde{\ell}_\xi^{(i)}(t)\cap\mathcal{C}^{(i)}_\xi$.
\end{proof}

The proof of Theorem \ref{theoremcurvature} will consist in establishing the following
\begin{thm}\label{theoremcurves}
For the symplectic isomorphism ${\rm T}:\mathcal{C}^{(2)}_\xi\rightarrow\mathcal{C}^{(\epsilon)}_\xi$ defined below,
\begin{equation}\label{relatingcurves}
{\rm T}\ell_\xi^{(2)}\bigl(g(t)\bigr)~=~\ell_\xi^{(\epsilon)}\bigl(t\bigr),
\end{equation}
where $g(t)=\frac{1}{H_2(\xi)\sigma}\ln(1+\epsilon\sigma H_\epsilon(\xi)t)$.
\end{thm}
\noindent {\it Definition of the map ${\rm T}$:} From $(\ref{decomposition})$, any $X\in\mathcal{C}^{(2)}_\xi$ may be uniquely expressed as $X=Z+aC_\xi+b(S_\epsilon)_\xi$, with $Z\in\mathcal{C}^{(\epsilon)}_\xi$. We define ${\rm T} (X)=Z$. To see that it is symplectic, we first note that we must have $b=0$. Indeed, from $(\ref{contactplane})$ we have $\omega(X,C_\xi)=\omega(Z,C_\xi)=0$, hence $b\,\omega((S_\epsilon)_\xi,C_\xi)=0$. But, $\omega\bigl((S_\epsilon)_\xi,C_\xi\bigr)=-F_\epsilon(\xi)\neq 0$, so
$b=0$. Now, the fact that ${\rm T}$ is symplectic is an immediate consequence of $(\ref{contactplane})$.
\\\\
In order to prove $(\ref{relatingcurves})$, we start by showing that the hypothesis $(\ref{homothetic})$ implies the following relation between the flows $\psi_t^{S_2},\psi_t^{S_\epsilon},\psi_t^{S_0}$:
\begin{prop}\label{propflow}
We have that $\psi_t^{S_2}=\psi_t^{S_0}\circ\psi_{f(t)}^{S_\epsilon}$, where \[f(t)=\begin{cases}
(\epsilon/\sigma)({\rm e}^{\sigma t}-1)& \text{if $\sigma\not=0$,}\\
t& \text{if $\sigma=0$,}
\end{cases}
\]
for $t$ in a neighborhood of $0$.
\end{prop}
\noindent For the proof of the above proposition, as it is clear that $(\ref{homothetic})$ also holds for the reverse metric $\tilde{F}$, it follows from $(\ref{flowS1})$ that
\begin{equation}
H_\epsilon\circ\psi_t^{S_0}~=~{\rm e}^{\sigma t}H_\epsilon,
\end{equation}
 at least in a neighborhood of $t=0$ (and whenever the flow $\psi_t^W$ is well-defined).
By taking the derivative of this relation at $t=0$, and recalling that $DH_\epsilon(X)=\omega(X,S_\epsilon)$, we get $\omega(S_0,S_\epsilon)=\sigma H_\epsilon$.
Hence, as $\omega(S_0,S_\epsilon)$ is equal to the Poisson bracket $\{H_\epsilon,H_0\}$, it follows that
\begin{equation}\label{bracket}
[S_0,S_\epsilon]~=~\sigma S_\epsilon.
\end{equation}
We can now apply the following general lemma to prove Proposition \ref{propflow} (for a proof, see \cite{ADH}):
\begin{lemma}\label{lemmaflow}
Let $X_0$ and $X_1$ be vector fields on some manifold, and let $Z$ be the time-dependent vector field defined by $Z_t=(\psi_t^{X_1})^\ast X_0$. Then, $\psi_t^{X_0+X_1}=\psi_t^{X_1}\circ\psi_t^Z$, where, for each $x$, $\psi_t^Z(x)$ is the evaluation at $t$ of the sotution to $\dot{x}(t)=Z_{(t,x(t))}$, $x(0)=x$. \end{lemma}
\begin{proof}[Proof of Proposition \ref{propflow}] From $(\ref{sumvectorfields})$ and the lemma above, $\psi_t^{S_2}=\psi_t^{S_0}\circ\psi_t^Z$, where $Z_t=\epsilon(\psi_t^{S_0})^\ast S_\epsilon$. Deriving in $t$ the last equation, and using $(\ref{bracket})$, we obtain, successively,
\begin{eqnarray}
\frac{d}{dt}(\psi_t^{S_0})^\ast S_\epsilon & = & (\psi_t^{S_0})^\ast[S_0,S_\epsilon]\nonumber\\
& = & \sigma(\psi_t^{S_0})^\ast S_\epsilon.\nonumber
\end{eqnarray}
Hence, $Z_t=\epsilon{\rm e}^{\sigma t}S_\epsilon$. It is now immediate that $\psi_t^Z=\psi_{\frac{\epsilon}{\sigma}({\rm e}^{\sigma t}-1)}^{S_\epsilon}$.

\end{proof}


\begin{proof}[Proof of Theorem \ref{theoremcurves}]
It follows from Proposition \ref{propflow}, and from the fact that $\psi_t^{S_0}$ preserves the fibers of $\tau:T^\ast M\rightarrow M$, that
\begin{equation}\nonumber
\tilde{\ell}_\xi^{(2)}(t)~=~\tilde{\ell}_\xi^{(\epsilon)}(f(t)).
\end{equation}
Therefore, having in mind (\ref{eq3}) and the definition of ${\rm T}$, we immediately get
\begin{equation}\nonumber
{\rm T}(\ell_\xi^{(2)}\bigl(t/H_2(\xi)\bigr))=\ell_\xi^{(\epsilon)}\bigl(f(t)/H_\epsilon(\xi)\bigr).
\end{equation}
This, in turn, is equivalent to (\ref{relatingcurves}).
\end{proof}

\noindent In order to apply Proposition \ref{propfanning} (at $t=0$) to (\ref{relatingcurves}), we first observe that
\begin{equation}
\frac{1}{2}\{g(t),t\}|_{t=0}=\frac{1}{4}\sigma^2H_\epsilon(\xi)^2~~~{\rm and}~~~\dot{g}(0)^2=\frac{H_\epsilon(\xi)^2}{H_2(\xi)^2}.
\end{equation}
\noindent Therefore, it follows from Proposition \ref{propfanning}, and the definition of flag curvature, that
\begin{equation}\nonumber
K_{\hat{F}}\Bigl(\mathscr{L}_{\hat{F}}^{-1}(\xi)\hspace{0.05cm},\hspace{0.05cm}{\rm span}\{\mathscr{L}_{\hat{F}}^{-1}(\xi),w\}\Bigr)~=~K_{{F_\epsilon}}\Bigl(\mathscr{L}_{F_\epsilon}^{-1}(\xi)\hspace{0.05cm},\hspace{0.05cm}{\rm span}\{\mathscr{L}_{F_\epsilon}^{-1}(\xi),\tilde{w}\}\Bigr)-\frac{1}{4}\sigma^2,
\end{equation}
where $\tilde{w}$ is the image of $w$ under the map
\begin{equation}\label{map}
(D_f\mathscr{L}_{F_\epsilon}(u))^{-1}\circ{\rm T}\circ \bigl(D_f\mathscr{L}_{\hat{F}}(v)\bigr)~:~\ker g_u^{\hat{F}}(u\hspace{0.05cm},\hspace{0.05cm}\cdot)\longrightarrow\ker g^{F_\epsilon}_v(v\hspace{0.05cm},\hspace{0.05cm}\cdot)
\end{equation}
 and $v:=\mathscr{L}_{F_\epsilon}^{-1}(\xi)$. Theorem \ref{theoremcurvature} now follows from (\ref{remarkreverse2}) and the following lemma.
\begin{lemma}
The map $(\ref{map})$ is multiplication by $\epsilon H_\epsilon(\xi)/H_2(\xi)$.
\end{lemma}
\begin{proof}
Via the canonical identification $\mathcal{V}_\xi T^\ast M=T_{\tau(\xi)}^\ast M$, $C_\xi$ corresponds to $\xi$ and, from Proposition \ref{propinverse},  $\mathcal{L}_\xi^{(\epsilon)}$ and $\mathcal{L}_\xi^{(2)}$ correspond to $\ker\mathscr{L}_{F_\epsilon}^{-1}(\xi)$ and $\ker\mathscr{L}_{\hat{F}}^{-1}(\xi)$, respectively. It follows that ${\rm T}\mid_{\mathcal{L}_\xi^{(2)}}:\mathcal{L}_\xi^{(2)}\rightarrow\mathcal{L}_\xi^{(\epsilon)}$ is the restriction to $\ker\mathscr{L}_{\hat{F}}^{-1}(\xi)$ of the projection map
$\ker\mathscr{L}_{F_\epsilon}^{-1}(\xi)\oplus{\rm span}\{\xi\}\rightarrow\ker\mathscr{L}_{F_\epsilon}^{-1}(\xi)$, and hence
\begin{equation}\label{eq100}
{\rm T}(\eta)~=~\eta-\frac{\mathscr{L}_{F_\epsilon}^{-1}(\xi)(\eta)}{\mathscr{L}_{F_\epsilon}^{-1}(\xi)(\xi)}\xi,
\end{equation}
for all $\eta\in\ker\mathscr{L}_{\hat{F}}^{-1}(\xi)$. Taking the fiber derivative of $(\ref{legendrerelations})$, we get, for all $\zeta\in T_{\tau(\xi)}^\ast M$,
\begin{eqnarray}
D_f\mathscr{L}_{F_\epsilon}^{-1}(\xi)(\zeta) & = &   \frac{\epsilon H_\epsilon(\xi)}{H_2(\xi)}D_f\mathscr{L}_{\hat{F}}^{-1}(\xi)(\zeta)-
\frac{\epsilon H_\epsilon(\xi)}{H_2(\xi)^3}\mathscr{L}_{\hat{F}}^{-1}(\xi)(\zeta)\mathscr{L}_{\hat{F}}^{-1}(\xi)\nonumber\\
& & +\frac{1}{H_\epsilon(\xi)^2}\mathscr{L}_{F_\epsilon}^{-1}(\xi)(\zeta)\mathscr{L}_{F_\epsilon}^{-1}(\xi).\nonumber
\end{eqnarray}
Letting, in the above equality, $\zeta$ be the right hand side of $(\ref{eq100})$, and using that $\mathscr{L}_{F_\epsilon}^{-1}(\xi)(\zeta)=\mathscr{L}_{\hat{F}}^{-1}(\xi)(\eta)=0$, $\mathscr{L}_{F_\epsilon}^{-1}(\xi)(\xi)=H_\epsilon(\xi)^2$, $\mathscr{L}_{\hat{F}}^{-1}(\xi)(\xi)=H_2(\xi)^2$ and $D_f\mathscr{L}_{\hat{F}}^{-1}(\xi)(\xi)=\mathscr{L}_{\hat{F}}^{-1}(\xi)$, we get $D_f\mathscr{L}_{F_\epsilon}^{-1}(\xi)(\zeta)=\epsilon (H_\epsilon(\xi)/H_2(\xi))D_f\mathscr{L}_{\hat{F}}^{-1}(\xi)(\eta)$ and the lemma follows.
\end{proof}

\begin{proof}[Proof of Theorem \ref{geodesicflow}]
 The expression for the geodesics follows easily from Proposition \ref{propflow}. Indeed, observe that the geodesics of the reverse pseudo-Finsler metric $\tilde{F}$ can be expressed as $\tilde{\gamma}(t)=\gamma(-t)$, where $\gamma$ is a geodesic of $F$. For the last claim, we need to prove that if $\hat{\gamma}$ is always a geodesic for the straight or for the reverse translation, but it cannot change from one to the other. Let us show first  $\hat{\gamma}$ is $\hat{F}$-unit. We know that  $\dot{\hat{\gamma}}(0)=W+\dot\gamma(0)$, $\dot{\hat{\gamma}}(t)=W+\dot\gamma(t)+(\psi_t^W)^*(\dot f(t) \dot\gamma(f(t)))$ and
\[F((\psi_t^W)^*(\dot f(t) \dot\gamma(f(t))))=e^{\sigma t} ((\psi_t^W)^*F( \dot\gamma(f(t)))=F(\dot\gamma(t))=1.\]
This implies that $\hat\gamma$ is an $\hat F$-unit curve. In order to apply Proposition \ref{translatingF} we need to show that if $v(t)=(\psi_t^W)^*(\dot f(t) \dot\gamma(f(t)))$, then
$g_{v(t)}(v(t),v(t)+W)$ has always the same sign (it cannot be zero):
\begin{align*}
g_{v(t)}(v(t),v(t)+W)&=1+g_{v(t)}(v(t),W)=1+e^{\sigma t}g_{(\psi_t^W)^*( \dot\gamma(f(t)))}((\psi_t^W)^*( \dot\gamma(f(t))),W)\\
&=1+e^{-\sigma t} g_{\dot\gamma(f(t))}(\dot\gamma(f(t)),W).
\end{align*}
Moreover, using that $W$ is a homothetic field, we deduce that $g_{\dot\gamma(t)}(\dot\gamma(t),W)=g_v(v,W)-\sigma t$. Substituting above we finally get
\[g_{v(t)}(v(t),v(t)+W)=e^{-\sigma t}(g_v(v,W)+1),\]
which concludes.
\end{proof}

\section{Conclusions and consequences}
 The first implication of Theorem \ref{theoremcurvature} is that we can reobtain all the Randers and Kropina metrics of constant flag curvature (see \cite{BCS04,Xia13,YoSa14,YoOk12}). In fact, all the metrics that we obtain are geodesically complete, since homothetic vector fields are complete in the spaces of constant curvature (see Remark \ref{geodesicflow}). In particular, the examples of Randers manifolds with constant flag curvature in \cite{BCS04} with the homothetic vector field with norm bigger than one in some subset can be extended to conic Finsler metrics defined in the whole manifold which are geodesically complete. Indeed, this metric is given by \eqref{Zermelo2} for $\varepsilon=1$ and it is defined in
\[A=\{v\in TM: \text{s. t. for $\pi(v)$, } g(W,W)<1\}\cup \{v\in TM: g(v,W)>0;h(v,v)>0\}.\]
Recall that Proposition \ref{transZer} ensures that if $g$ is a Riemannian metric, then this metric has positive definite fundamental tensor, namely, it is a conic Finsler metric.

 Observe that Theorem \ref{thmmatsumoto} and Proposition \ref{Randers-Kro} ensure that pseudo-Randers-Kropina metrics are characterized as the pseudo-Finsler metrics with Matsumoto tensor trivially equal to zero. It is expectable that all the pseudo-Randers-Kropina metrics with constant flag curvature are those obtained as translation  of a pseudo-Riemannian metric $g$ with constant sectional curvature by a homothetic vector field of $g$ (see Proposition \ref{Randers-Kro}). To prove this, it is enough to extend the computations in \cite{BR03,BCS04} for pseudo-Randers metrics and those in \cite{Xia13,YoSa14,YoOk12} for pseudo-Kropina metrics. This seems very likely because in these computations, the signature of the metric does not seem essential. Even if these computations do not hold in the subset of points where the pseudo-Randers-Kropina metric passes from being pseudo-Kropina to being pseudo-Randers, this subset of points is irrelevant for this computation because it has empty interior.


\begin{thebibliography}{10}

\bibitem{Ah}
{\sc S.~Ahdout}, {\em Fanning curves of {L}agrangian manifolds and geodesic
  flows}, Duke Math. J., 59 (1989), pp.~537--552.

\bibitem{AD}
{\sc J.~C. {\'A}lvarez~Paiva and C.~E. Dur{\'a}n}, {\em Geometric invariants of
  fanning curves}, Adv. in Appl. Math., 42 (2009), pp.~290--312.

\bibitem{ADH}
{\sc J.~C. {\'A}lvarez~Paiva, C.~E. Dur{\'a}n, and H.~Vit\'orio}, {\em
  Geometric invariants of fanning curves and the differential geometry of
  sprays and {F}insler metrics}, In preparation.

\bibitem{BR03}
{\sc D.~Bao and C.~Robles}, {\em On {R}anders spaces of constant flag
  curvature}, Rep. Math. Phys., 51 (2003), pp.~9--42.

\bibitem{BCS04}
{\sc D.~Bao, C.~Robles, and Z.~Shen}, {\em Zermelo navigation on {R}iemannian
  manifolds}, J. Differential Geom., 66 (2004), pp.~377--435.

\bibitem{BiJa11}
{\sc L.~Biliotti and M.~{\'A}. Javaloyes}, {\em {$t$}-periodic light rays in
  conformally stationary spacetimes via {F}insler geometry}, Houston J. Math.,
  37 (2011), pp.~127--146.

\bibitem{CJS14}
{\sc E.~Caponio, M.~A. Javaloyes, and M.~S{\'a}nchez}, {\em Wind {F}inslerian
  structures: from {Z}ermelo's navigation to the causality of spacetimes},
  arXiv:1407.5494 [math.DG].

\bibitem{Foulon}
{\sc P.~Foulon}, {\em G\'eom\'etrie des \'equations diff\'erentielles du second
  ordre}, Ann. Inst. H. Poincar\'e Phys. Th\'eor., 45 (1986), pp.~1--28.

\bibitem{Grifone}
{\sc J.~Grifone}, {\em Structure presque-tangente et connexions. {I}}, Ann.
  Inst. Fourier (Grenoble), 22 (1972), pp.~287--334.

\bibitem{HuMo11}
{\sc L.~Huang and X.~Mo}, {\em On geodesics of {F}insler metrics via navigation
  problem}, Proc. Amer. Math. Soc., 139 (2011), pp.~3015--3024.

\bibitem{JaSan11}
{\sc M.~A. Javaloyes and M.~S{\'a}nchez}, {\em On the definition and examples
  of {F}insler metrics}, Ann. Sc. Norm. Super. Pisa Cl. Sci., XIII (2014),
  pp.~813--858.

\bibitem{JavSoa13}
{\sc M.~A. Javaloyes and B.~Soares}, {\em {G}eodesics and {J}acobi fields of
  pseudo-{F}insler manifolds}, arXiv:1401.8149 [math.DG],  (2013).

\bibitem{kostelecky11}
{\sc V.~A. Kosteleck{\'y}}, {\em Riemann-finsler geometry and lorentz-violating
  kinematics}, Physics Letters B, 701 (2011), pp.~137--143.

\bibitem{MaHo78}
{\sc M.~Matsumoto and S.-i. H{\=o}j{\=o}}, {\em A conclusive theorem on
  {$C$}-reducible {F}insler spaces}, Tensor (N.S.), 32 (1978), pp.~225--230.

\bibitem{MoHu07}
{\sc X.~Mo and L.~Huang}, {\em On curvature decreasing property of a class of
  navigation problems}, Publi. Math. Debrecen, 71 (2007), pp.~141--163.

\bibitem{MoHu10}
{\sc X.~Mo and L.~Huang}, {\em On characterizations
  of {R}anders norms in a {M}inkowski space}, Internat. J. Math., 21 (2010),
  pp.~523--535.

\bibitem{NoSa94}
{\sc K.~Nomizu and T.~Sasaki}, {\em Affine differential geometry}, vol.~111 of
  Cambridge Tracts in Mathematics, Cambridge University Press, Cambridge, 1994.
\newblock Geometry of affine immersions.

\bibitem{Rob07}
{\sc C.~Robles}, {\em Geodesics in {R}anders spaces of constant curvature},
  Trans. Amer. Math. Soc., 359 (2007), pp.~1633--1651 (electronic).

\bibitem{Sh03}
{\sc Z.~Shen}, {\em {Finsler metrics with {$\mathbf K=0$} and {$\mathbf
  S=0$}}}, Canad. J. Math., 55 (2003), pp.~112--132.

\bibitem{Henrique}
{\sc H.~Vit\'orio}, {\em Geometria de curvas fanning e de suas redu\c c\~oes
  simpl\'eticas}, P.h.d. thesis, Universidade estadual de Campinas,  (2010).

\bibitem{Xia13}
{\sc Q.~Xia}, {\em On {K}ropina metrics of scalar flag curvature}, Differential
  Geom. Appl., 31 (2013), pp.~393--404.

\bibitem{YoOk12}
{\sc R.~Yoshikawa and K.~Okubo}, {\em Constant curvature conditions for
  {K}ropina spaces}, Balkan J. Geom. Appl., 17 (2012), pp.~115--124.

\bibitem{YoSa14}
{\sc R.~Yoshikawa and S.~V. Sabau}, {\em Kropina metrics and {Z}ermelo
  navigation on {R}iemannian manifolds}, Geom. Dedicata, 171 (2014),
  pp.~119--148.

\bibitem{Ze31}
{\sc E.~Zermelo}, {\em \"uber das navigationsproblem bei ruhender oder
  ver\"anderlicher windverteilung.}, Z. Angew. Math. Mech., 11 (1931),
  pp.~114--124.

\bibitem{Ziller}
{\sc W.~Ziller}, {\em Geometry of the {K}atok examples}, Ergodic Theory Dynam.
  Systems, 3 (1983), pp.~135--157.

\end{thebibliography}
\end{document}